%% file: _natbib_paper_formatted_for_submission.tex
\documentclass[final, nomarks]{dmtcs-episciences}

\usepackage{amsmath,amssymb,amsthm}
\usepackage{mathtools}
\usepackage{caption, subcaption}
\usepackage{color}
\usepackage[dvipsnames]{xcolor}
\usepackage{tikz}
\usepackage{float}
\usepackage{array}

\newtheorem{thm}{Theorem}

\newtheorem{lem}{Lemma}
\newtheorem{rem}{Remark}
\newtheorem{prop}{Proposition}

\theoremstyle{definition}

\newcommand{\cl}{\mathcal}
\newcommand{\parx}{\frac{\partial}{\partial x}}
\newcommand{\parz}{\frac{\partial}{\partial z}}
\newcommand{\veps}{\varepsilon}
\newcommand{\dg}{\dagger}

\newcommand{\mcI}{\mathcal{I}}
\newcommand{\mcS}{\mathcal{S}}
\newcommand{\mcP}{\mathcal{P}}

\DeclareMathOperator{\premax}{premax}

\numberwithin{equation}{section}

\author[Nathan Britt and Nicholas Beaton]{Nathan Britt \and Nicholas Beaton}

\title{Completing the enumeration of inversion sequences avoiding triples of relations}

\affiliation{School of Mathematics and Statistics, The University of Melbourne, Melbourne, Australia}

\usepackage[square,numbers]{natbib}
\usepackage{doi}
\RequirePackage[capitalize,nameinlink,noabbrev]{cleveref}
\Crefname{subsection}{Section}{Sections}
\Crefname{lem}{Lemma}{Lemmas}
\Crefname{prop}{Proposition}{Propositions}
\Crefname{table}{Table}{Tables}

\begin{document}

\publicationdata{vol. 28:1, Permutation Patterns 2025}
{2026}
{8}
{10.46298/dmtcs.17203}
{2025-12-29; 2025-12-29; 2026-08-04; 2026-09-30}
{2026-10-07}

\maketitle
\begin{abstract}
An inversion sequence of length $n$ is an integer sequence $(a_1, \ldots, a_n)$ such that $0 \le a_i < i$ for all~$i$. The study of pattern-avoiding inversion sequences was initiated in 2015 by Mansour and Shattuck and in 2016 by Corteel, Martinez, Savage and Weselcouch. Martinez and Savage later defined a new type of pattern, a triple of binary relations, of which there are currently 14 uncounted avoidance classes. We complete the enumeration for all of these classes using generating tree methods ``growing on the left'' and ``growing on the right''. For many of these classes we are able to find algebraic generating functions. We also discuss the asymptotic behaviour of the counting sequences.
\end{abstract}

\section{Introduction}

The introduction of pattern avoidance in permutations is often attributed to Knuth in the first volume of his seminal work \emph{The Art of Computer Programming} in 1968 \cite{knuth_taocp}. There, he was interested in permutations which could be sorted using a single stack, which turn out to be exactly those which avoid the pattern 231. A more systematic study was initiated by Simion and Schmidt in 1985 \cite{simion_restricted_1985}, and there have been many works since which relate pattern-avoiding permutations to combinatorics, computer science, algebra, statistical mechanics and other fields. Kitaev's 2011 monograph \cite{kitaev_patterns_2011} gives a thorough treatment to the subject.

An \emph{inversion sequence} of length $n$ is a sequence $(a_1,\dots,a_n)$ of integers which satisfy $0 \leq a_i < i$ for all $i$. If $\mcI_n$ is the set of such sequences then clearly $I_n = |\mcI_n| = n!$. The name comes from a simple bijection with the set of permutations $\mcS_n$
\begin{equation}
    \Phi: \mcS_n \to \mcI_n, \qquad \Phi(\pi_1\dots\pi_n) = (a_1,\dots,a_n) \text{ where } a_i = |\{j : j<i \text{ and } \pi_j > \pi_i\}|.
\end{equation}

One can consider inversion sequences which avoid patterns in much the same way as pattern-avoiding permutations. We say that a sequence of non-negative integers $\sigma = \sigma_1\cdots\sigma_k$ is an \emph{inversion pattern} (or just \emph{pattern}) if $\sigma$ contains every digit from $0$ to $\max(\sigma)$. We also define the \emph{reduction} (or \emph{standardisation}) of a sequence of non-negative integers $w = w_1\cdots w_n$ to be the sequence obtained by replacing the smallest entries of $w$ by 0, the next-smallest entries by 1, and so on. That is, the reduction of $w$ is the pattern whose entries are in the same relative order as $w$. For example, the reduction of $1035573$ is $1023342$.

If an inversion sequence $a$ has a (not necessarily consecutive) subsequence $a_{i_1}a_{i_2}\cdots a_{i_k}$ whose reduction is a pattern $\sigma$, we say that $a$ \emph{contains the pattern} $\sigma$. If no such subsequence exists then we say $a$ \emph{avoids} $\sigma$. For example, the inversion sequence $002140348$ contains the pattern $011$ (four times, in the two occurrences of the subsequence $044$ and the subsequences $244$ and $144$), but does not contain the pattern $110$.

The study of pattern-avoiding inversion sequences was initiated by Mansour and Shattuck in 2015 \cite{mansour_pattern_2015} and Corteel, Martinez, Savage and Weselcouch in 2016 \cite{corteel_patterns_2016}. These and most subsequent works have focused on patterns of length 3, of which there are 13:
\[ \mcP = \{000,\, 001,\, 010,\, 011,\, 012,\, 021,\, 100,\, 101,\, 102,\, 110,\, 120,\, 201,\, 210\}.\]
(Note that because inversion sequences, and hence patterns, can have repeated digits, there are many more patterns of a given size for inversion sequences than there are for permutations.)  If $S$ is a set of patterns then let $\mcI_n(S)$ denote the set of length $n$ inversion sequences which avoid all patterns in $S$, and let $I_n(S) = |\mcI_n(S)|$. We will generally abuse notation and write $\mcI_n(\sigma_1, \sigma_2,\dots)$ and so on. By convention we will set $I_0(S) = 1$ for any $S$. We also write $\mcI(S) = \bigcup_{n \geq 0} \mcI_n(S)$.

Most of the cases of $\mcI_n(\sigma)$ for $\sigma \in \mcP$ (i.e.\ avoidance of a single length 3 pattern) were treated by \citet{mansour_pattern_2015} and \citet{corteel_patterns_2016}. The solutions are variously in the forms of simple closed solutions (e.g.\ $I_n(001) = 2^{n-1}$), other well-known combinatorial sequences (e.g.\ $I_n(021)$ is the $n^\text{th}$ large Schr\"oder number), algebraic generating functions, or functional equations involving several variables which yield polynomial-time enumerations. The exceptions are the patterns 010 and 100 which are covered by \citet{testart_completing_2024} and \citet{kotsireas_algorithmic_2024} respectively. (At this point we highly recommend Pantone's 2024 article \cite{Pantone} which gives a nice summary of the state of the art.) There are two pairs of patterns which have the same counting sequence, namely 101 / 110 and 201 / 210, and hence there are 11 ``Wilf classes'' for single patterns of length 3.

For pairs of patterns of length 3, there are 48 Wilf classes among the $\binom{13}{2}=78$ pairs, as determined by Yan and Lin \cite{yan_inversion_2020}. They collected many known results and found connections with existing combinatorial objects, as well as finding many new solutions. Subsequent works \cite{beaton_enumerating_2019, bouvel_semi-baxter_2018, testart_completing_2024, testart_generating_2025, kotsireas_algorithmic_2024, chen_bijection_2024} have further studied pairs of patterns of length 3, and solutions of one form or another are now known for all cases. Callan, Jel\'inek and Mansour \cite{callan_inversion_2023} undertook a systematic study of triples of patterns of length 3, determining that there are probably 137 (and at most 139) Wilf classes among the $\binom{13}{3} = 286$ triples. The counting sequences for many of these classes are known but there still remain a number of open problems. Callan and Mansour \cite{callan_quadruples_2023} studied quadruples of patterns of length 3, determining that there are between 212 and 215 Wilf classes.

A different type of inversion pattern was introduced by Martinez and Savage in 2018 \cite{martinez_patterns_2018}. They considered ``triples of binary relations'' $(\rho_1, \rho_2, \rho_3)$ and the corresponding avoidance classes $\mcI_n(\rho_1,\rho_2,\rho_3)$, where we say
\[ a \in \mcI_n(\rho_1,\rho_2,\rho_3) \text{ if there is no } i < j < k \text{ with } a_i \rho_1 a_j,\, a_j \rho_2 a_k, \text{ and } a_i \rho_3 a_k. \]
Here we have $\rho_i \in \{<, >, \leq, \geq, =, \neq, -\}$ where `$-$' means no restriction is applied, i.e. $x\! -\! y$ is true for any $x,y$. Avoidance of any given triple of binary relations always corresponds to avoidance of a set $S$ of patterns of length 3, but the size of $S$ depends on the particular triple. For example, $\mcI_n(<,>,<) = \mcI_n(021)$ while $\mcI_n(\geq, \geq, \geq) = \mcI_n(000, 100, 110, 210)$. Here we follow the notation of Martinez and Savage and say that if two triples of relations have the same avoidance \emph{sets} then they are \emph{equivalent}, while if they have the same counting sequences they are \emph{Wilf-equivalent}.

There are $7^3 = 343$ possible triples of relations, which Martinez and Savage found to consist of 98 equivalence classes. They conjectured that these fall into 63 Wilf-equivalence classes, which has subsequently been verified (see below for details). Of these 63 Wilf classes, they found:
\begin{itemize}
    \item 5 counting sequences are eventually constant;
    \item 30 counting sequences appeared in the Online Encyclopedia of Integer Sequences (OEIS)~\cite{oeis};
    \item 28 counting sequences were not in the OEIS.
\end{itemize}
All but a handful of the 30 Wilf classes whose counting sequences already appeared in the OEIS were dealt with by \citet{martinez_patterns_2018}. Many of these results had already appeared in the literature, mostly relating to pattern-avoiding permutations or regular pattern-avoiding inversion sequences. The remaining open problems (either proving Wilf-equivalence or counting sequences) were solved by Cao, Jin and Lin in 2019 \cite{cao_enumeration_2019} and Lin in 2020 \cite{lin_patterns_2020}.

The 28 Wilf classes which were not yet in the OEIS were mostly left as open problems in \cite{martinez_patterns_2018}. Many others have been addressed subsequently by several authors. We summarise what we believe to be the current state of affairs in \cref{table:main1,table:main2}.

\begin{table}
\centering
{\small
    \begin{tabular}{c l c c p{4cm}}
    Triple & Patterns avoided & OEIS & Index in \cite{martinez_patterns_2018} & Results \\[1mm]
    \hline\\[-2mm]
     $(-,\neq,\geq)$ & $(010,101,110,120,201,210)^*$ & \href{https://oeis.org/A279553}{A279553} & 663A & Wilf-eq with 663B \cite{martinez_patterns_2018}; \cref{ssec:663A} (alg:3) \\[1mm]
    $(\neq,-,\geq)$ & $(010,100,101,120,201,210)^*$ & \href{https://oeis.org/A279553}{A279553} & 663B & as above\\[1mm]
    $(\neq, \neq, \geq)$ & $(010,101,120,201,210)^*$ & \href{https://oeis.org/A279554}{A279554} & 733 & \cref{ssec:733} (alg:4)\\[1mm]
    $(\leq, >, \neq)$ & $(021,110,120)$ & \href{https://oeis.org/A279557}{A279557} & 805 & counted by \citet{martinez_patterns_2018} (alg:2)\\[1mm]
    $(>,-,\neq)$ & $(100,102,201,210)$ & \href{https://oeis.org/A279560}{A279560} & 1016 & counted by \citet{martinez_patterns_2018}; \cref{ssec:1016} (alg:2) \\[1mm]
    $(>,\neq,-)$ & $(101,102,201,210)$ & \href{https://oeis.org/A279561}{A279561} & 1079A & counted by \citet{martinez_patterns_2018} (alg:2) \\[1mm]
    $(<,>,\neq)$ & $(021,120)$ & \href{https://oeis.org/A279561}{A279561} & 1079B & counted by \citet{yan_inversion_2020} (alg:2)\\[1mm]
    $(>,\leq,\neq)$ & $(100,102,201)$ & \href{https://oeis.org/A279562}{A279562} & 1176 & \cref{ssec:1176} (alg:2) \\[1mm]
    $(>,\neq,\neq)$ & $(102,201,210)$ & \href{https://oeis.org/A279563}{A279563} & 1253 & \cref{ssec:1253} (alg:2) \\[1mm]
    $(-,-,>)$ & $(100,110,120,201,210)$ & \href{https://oeis.org/A279565}{A279565} & 1420 & unreferenced solution in OEIS; \cref{ssec:1420} (alg:3) \\[1mm] 
    $(>,<,\leq)$ & $(102,201)$ & \href{https://oeis.org/A279566}{A279566} & 1465 & counted by \citet{testart_completing_2024} (alg:2) \\[1mm]
    $(-,\neq,>)$ & $(110,120,201,210)$ & \href{https://oeis.org/A279568}{A279568} & 1833A & Wilf-eq with 1833B \cite{martinez_patterns_2018}; \cref{ssec:1833A} (alg:6) \\[1mm]
    $(\neq,-,>)$ & $(100,120,201,210)$ & \href{https://oeis.org/A279568}{A279568} & 1833B & as above\\[1mm]
    $(\neq, \neq, >)$ & $(120,201,210)$ & \href{https://oeis.org/A279572}{A279572} & 2468 & counted by \citet{mansour_sorting_2025} (alg:6) \\[2mm]
    \end{tabular}
}
    \caption{The classes from \citet{martinez_patterns_2018} whose counting sequences did not appear in the OEIS at the time of publication, and whose generating functions are known to be algebraic. The terms in the fourth column correspond to $I_7(\cdot)$ (this being the first term which can be used to unambiguously distinguish between different classes). The $^*$ superscript in the second column indicates classes where the wrong set of patterns was cited by \citet{martinez_patterns_2018} and the OEIS \cite{oeis} (the enumeration itself was correct; the descriptions have now been corrected in the OEIS). The (alg:$d$) in the final column indicates that the class has an algebraic generating function, with $d$ indicating the degree of the minimal polynomial.}
    \label{table:main1}
\end{table}

\begin{table}
\centering
{\small
    \begin{tabular}{c l c c p{4cm}}
    Triple & Patterns avoided & OEIS & Index in \cite{martinez_patterns_2018} & Results \\[1mm]
    \hline\\[-2mm]
    $(-, \geq, \geq)$ & $(000,010,100,110,120,210)^*$ & \href{https://oeis.org/A279544}{A279544} & 214 & \cref{ssec:214} \\[1mm]
    $(\leq,-,\geq)$ & $(000,010,110,120)$ & \href{https://oeis.org/A279551}{A279551} & 247 & \cref{ssec:247} \\[1mm]
    $(-,\geq,=)$ & $(000,010)$ & \href{https://oeis.org/A279552}{A279552} & 345 & counted by \citet{testart_completing_2024}\\[1mm]
    $(-,>,\geq)$ & $(010,110,120,210)$ & \href{https://oeis.org/A279555}{A279555} & 746A & Wilf-eq with 746B \cite{martinez_patterns_2018}; counted by Asinowski et. al. \cite{asinowski_patterns_2025}\\[1mm]
    $(\neq,\geq,\geq)$ & $(010,100,120,210)$ & \href{https://oeis.org/A279555}{A279555} & 746B & as above\\[1mm]
    $(\leq, \neq, \geq)$ & $(010,110,120)$ & \href{https://oeis.org/A279556}{A279556} & 759 & \cref{ssec:759} \\[1mm]
    $(\neq, >, \geq)$ & $(010,120,210)$ & \href{https://oeis.org/A279558}{A279558} & 830 & \cref{ssec:830} \\[1mm]
    $(<,-,\geq)$ & $(010,120)$ & \href{https://oeis.org/A279559}{A279559} & 845 & counted by \citet{testart_completing_2024}\\[1mm]
    $(-, >, =)$ & $(010)$ & \href{https://oeis.org/A263779}{A263779} & 979 & counted by \citet{testart_completing_2024}\\[1mm]
    $(\geq,=,-)$ & $(000,100)$ & \href{https://oeis.org/A279564}{A279564} & 1267 & counted by \citet{testart_completing_2024} \\[1mm]
    $(-,\geq,>)$ & $(100,110,120,210)$ & \href{https://oeis.org/A279567}{A279567} & 1509 & \cref{ssec:1509} \\[1mm]
    $(-,>,>)$ & $(110,120,210)$ & \href{https://oeis.org/A279569}{A279569} & 1953A & Wilf-eq with 1953B \cite{martinez_patterns_2018}; \cref{ssec:1953A} \\[1mm]
    $(\neq,\geq,>)$ & $(100,120,210)$ & \href{https://oeis.org/A279569}{A279569} & 1953B & as above\\[1mm]
    $(\leq,>,>)$ & $(110,120)$ & \href{https://oeis.org/A279570}{A279570} & 2091 & counted by \citet{testart_completing_2024} \\[1mm] 
    $(>,\leq,\geq)$ & $(100,101,201)$ & \href{https://oeis.org/A279571}{A279571} & 2106 & \cref{ssec:2106} \\[1mm]
    $(\neq,>,>)$ & $(120,210)$ & \href{https://oeis.org/A279573}{A279573} & 2625 & counted by \citet{kotsireas_algorithmic_2024} \\[1mm]
    $(<,-,>)$ & $(120)$ & \href{https://oeis.org/A263778}{A263778} & 2803 & counted by \citet{mansour_pattern_2015} \\[1mm]
    $(>,=,-)$ & $(100)$ & \href{https://oeis.org/A263780}{A263780} & 3399 & counted by \citet{kotsireas_algorithmic_2024} \\[1mm]
    $(-,<,>)$ & $(201)$ & \href{https://oeis.org/A263777}{A263777} & 4306A & Wilf-eq with 4306B \cite{corteel_patterns_2016, mansour_pattern_2015}; counted by \citet{mansour_pattern_2015} \\[1mm]
    $(>,>,-)$ & $(210)$ & \href{https://oeis.org/A263777}{A263777} & 4306B & as above
    \end{tabular}
}
    \caption{%The classes from \cite{martinez_patterns_2018} whose counting sequences  did not appear in the OEIS at the time of publication. The terms in the fourth column correspond to $I_7(\cdot)$ (this being the first term which can be used to unambiguously distinguish between different classes). The $^*$ superscript in the second column indicates classes where the wrong set of patterns was cited in \cite{martinez_patterns_2018,oeis} (the enumeration itself was correct; the descriptions have now been corrected in the OEIS). Those classes in the top section with an (alg:$d$) in the final column have algebraic generating functions, with $d$ indicating the degree of the minimal polynomial. We believe the ones in the bottom section do not have algebraic generating functions.
    The classes from \citet{martinez_patterns_2018} whose counting sequences did not appear in the OEIS at the time of publication, and whose generating functions are not believed to be algebraic (or not known). Notation is as in \cref{table:main1}.}
    \label{table:main2}
\end{table}

In this paper we address all the 14 classes in \cref{table:main1,table:main2} which have not yet been solved. In all cases we derive a combinatorial generating tree which gives a polynomial-time algorithm for computing the number of inversion sequences which avoid the given triple of relations. For most classes we can further obtain a system of functional equations satisfied by a set of generating functions, and 7 of these are solvable via the kernel method (see \citet{prodinger_kernel_2004}), yielding an algebraic solution. For all the classes where we do not obtain an algebraic generating function, we believe (based on guessing software and/or asymptotics) that the generating functions are not algebraic and probably not D-finite. All enumeration results in this paper have been checked against brute-force enumeration, as well as existing series in the OEIS. \textsc{Mathematica} was used to assist with solving the algebraic cases -- a notebook which goes through the calculations can be found at the second author's website \href{www.nicholasbeaton.com}{www.nicholasbeaton.com}, and is also available as an ancillary file at \href{https://arxiv.org/abs/2512.21943}{arXiv:2512.21943}.

There are two broad strategies we employ here to count pattern-avoiding inversion sequences. These are:
\begin{itemize}
    \item ``Growing on the right'': Here one constructs a generating tree by tracking a set of statistics about a pattern-avoiding inversion sequence $a = (a_1,\dots,a_n)$ in order to determine which potential values can be appended to $a$ to give a valid longer inversion sequence.
    \item ``Growing on the left'': Again a tree is constructed by tracking a set of statistics, but this time the operation on $a$ involves increasing all non-zero entries of $a$ by 1, increasing some zero entries by 1 (exactly which depends on the patterns to be avoided), and then prepending a 0.
\end{itemize}
The name ``growing on the left'' was introduced by Testart \cite{testart_generating_2025}, although we believe the construction was first used by \citet{corteel_patterns_2016}. 

% \nick{maybe more introduction stuff here}
A \emph{generating tree} for a combinatorial class $\cl{C}$ consists of a set of ``labels'' and ``successions'' such that each object of $\cl{C}$ of size $n \in \mathbb{Z}_{\ge0}$ has one corresponding label at depth $n$ of the tree, and the successions describe how to obtain objects of size $n+1$ from the objects of size $n$. This technique eschews the need to track entire inversion sequences, and instead allows tracking of multiple statistics which determine how an inversion sequence may progress to other pattern-avoiding inversion sequences.

The efficiency of a generating tree depends on the number of distinct labels at each depth, where fewer labels requires less computation. 

The structure of this paper is as follows. In \cref{sec:right} we consider those classes from \cref{table:main1,table:main2} which are most naturally constructed by growing on the right. In \cref{sec:left} we study the classes for which growing on the left is simpler. In \cref{sec:asymptotics} we summarise the asymptotic behaviour of $I_n(S)$ for each $S$ considered in \cref{sec:right,sec:left}. Finally in \cref{sec:conclusion} we consider some open questions and possible avenues for future work.

\section{Inversion sequences grown on the right}\label{sec:right}

In this section we look at those classes from \cref{table:main1,table:main2} which are (what we found to be) most straightforward to construct when growing on the right. For each class, we present a succession rule which tracks (a subset of) the following statistics for each inversion sequence:
\begin{itemize}
    \item $\max(a) := \max\{a_i \;|\; i \in [1, n],\; a = (a_1, \ldots, a_n)\}$: the greatest value of $a$;
    \item $\premax(a) := \max\Big(\{a_i \;|\; i \in [1, n],\; a = (a_1, \ldots, a_n)\}\Big\backslash \{\max(a)\} \Big)$: the second-greatest value of $a$ (called the `premaximum').
\end{itemize}

Throughout, $a_i$ is the \emph{value} of $a = (a_1, \ldots, a_n)$ which occurs in \emph{position} $i$. Each tree has as its root the unique empty inversion sequence, defined to have length 0 and maximum 0. For the purposes of our succession rules, any inversion sequence with maximum value 0 has undefined premaximum, with the exception of class 2106 in \cref{ssec:2106} where such inversion sequences are defined with premaximum 0.

Define `unimodal' to describe inversion sequences of the form $a = a_1 \le \cdots \le a_i \ge \cdots \geq a_n$ for some $i \in [1, n]$.

\subsection{Class 1176: $(>,\leq,\neq) \equiv (100, 102, 201)$}
\label{ssec:1176}

% The set of patterns corresponds to the triple of relations $(>, \le, \ne)$. 
Here we construct a succession rule growing on the right which is based heavily on the succession rule for $\cl{I}(102, 201)$ by \citet{testart_completing_2024}. We will then adapt this rule for classes 1253 ($(>,\neq,\neq) \equiv (102,201,210)$, see \cref{ssec:1253}) and 1016 ($(>,-,\neq) \equiv (100, 102, 201, 210)$, see \Cref{ssec:1016})
and present the algebraic generating function for each of these classes.

The following result from \cite{testart_completing_2024} will be useful so we restate it here.
\begin{lem}[\citet{testart_completing_2024}]\label{p1p2p3_102-201}
An inversion sequence $a \in \cl{I}$ avoids the patterns 102 and 201 if and only if it can be factored as $a = p_1 \circ p_2 \circ p_3$ where:
\begin{itemize}
    \item $p_1$ is a non-decreasing inversion sequence such that $\max(p_1) < \max(a)$ (equivalently $p_1$ ends on a value below $max(a)$);
    \item $p_2$ is a sequence of values from $\{\max(a), \: \premax(a)\}$, beginning on the value $max(a)$;
    \item $p_3$ is a (possibly empty) non-increasing sequence such that $\max(p_3) < \premax(a)$ (equivalently $p_3$ begins on a value below $\premax(a)$).
\end{itemize}
\end{lem}
We modify this characterisation to fit inversion sequences which also avoid $100$. The pattern $100$ cannot occur in the non-decreasing $p_1$, but may be present in $p_2$ if $\premax(a)$ occurs more than once as $\left(\max(a), \premax(a), \premax(a)\right)$. It may also appear across $p_2$ and $p_3$ if any value in $p_3$ repeats: $\left(\max(a), k, k\right)$ for some $k < \premax(a)$.

\begin{prop}
An inversion sequence $a \in \cl{I}$ avoids the set of patterns $(100, 102, 201)$ if and only if it can be factored as $a = p_1 \circ q_2 \circ q_3$ where:
\begin{itemize}
    \item $p_1$ is as in \cref{p1p2p3_102-201};
    \item $q_2$ is a sequence of values from $\{\max(a),\;\premax(a)\}$, beginning on the value $\max(a)$ and such that $\premax(a)$ occurs no more than once;
    \item $q_3$ is a (possibly empty) decreasing sequence such that $\max(q_3) < \premax(a)$.
\end{itemize}
\end{prop}

We can now begin to describe our succession rule grown on the right. A further remark by \citet{testart_completing_2024} is that $a \in \cl{I}_n(102, 201)$ is unimodal if and only if it additionally avoids the pattern $101$, and that an occurrence of $101$ necessarily takes the form $\left(\max(a), \premax(a), \max(a)\right)$. The same is true for $\cl{I}(100, 102, 201)$. It is convenient to separately consider inversion sequences which avoid $(100, 102, 201)$ and $101$, and then those which contain $101$.

Define the following sets:
\begin{itemize}
    \item $\cl{A}_{n, h} := \{a \in \cl{I}_n(10) \;|\; a_n = h\}$: the set of non-decreasing inversion sequences of length $n$ ending on value $h$;
    \item $\cl{E}_{n, \ell}^{(1)} := \{a \in \cl{I}_n(100, 101, 102, 201)\; |\; a_n = \ell \;\text{and}\; a_n < \max(a)\}$: the set of unimodal inversion sequences of length $n$ which end on value $\ell$ and strictly descend after the final occurrence of the maximum.
\end{itemize}
For these and other sets we will sometimes omit one or all of the subscripts, for example just writing $\cl{A}_n$ or $\cl{E}^{(1)}$. Usually this should be interpreted as taking a union over all of the omitted parameters.

The structure for a $(100, 101, 102, 201)$-avoiding inversion sequence is simple: from left to right, values repeat or increase up to a maximum, and afterwards strictly decrease. Equivalently, an element of $\cl{A}_{n, h}$ may transition to an element $\cl{A}_{n+1, i}$ for $i \in [h, n]$ while repeating or increasing, or may transition to an element of $\cl{E}_{n+1, \ell}^{(1)}$ for $\ell < h$ when decreasing. An element of $\cl{E}_{n, \ell}^{(1)}$ may only continue to decrease, transitioning to $\cl{E}_{n+1, i}^{(1)}$ for $i < \ell$. We describe this behaviour in the following succession rule:

\begin{prop}
Give the label $(n, h)_a$ to each sequence in $\cl{A}_{n,h}$, and give the label $(\ell)_{e^{(1)}}$ to each sequence in $\cl{E}_{n, \ell}^{(1)}$. We have the following succession rule for $\cl{I}_n(100, 101, 102, 201)$:

\begin{equation}
\Omega_{(100, 101, 102, 201)}: \left\{\begin{aligned}
(0, 0)_a \\
(n, h)_a &\leadsto (n+1, i)_a &&\text{for}\; i \in [h, n] \\
    &\leadsto (i)_{e^{(1)}} &&\text{for}\; i \in [0, h-1] \\
(\ell)_{e^{(1)}} &\leadsto (i)_{e^{(1)}} &&\text{for}\; i \in [0, \ell-1].
\end{aligned}\right.
\label{eqn:rule100-101-102-201}\end{equation}
\end{prop}

We now describe $(100, 102, 201)$-avoiding inversion sequences which \emph{do} contain $101$. Recall that a $101$ pattern must take the form $\left(\max(a), \premax(a), \max(a)\right)$, which means it occurs entirely within the factor $q_2$ and $\premax(a)$, which corresponds to the 0 of $101$, occurs precisely once. It will be convenient to categorise inversion sequences according to \emph{how much} of the $101$ pattern they contain. Define the following sets:

\begin{itemize}
    \item $\cl{B}_{n, k} := \{a \in \cl{I}_n(10)\; |\; k \in [\premax(a),\; \max(a)-1]\}$: sequences where only the 1 in a 101 pattern has occurred. The value $k$ will represent the 0 of a 101 pattern, which is yet to occur. These sequences are also non-decreasing.
    \item $\cl{C}_{n, k} := \{(a_1,\dots,a_{n-1}) \circ (k) \;|\; (a_1,\dots,a_{n-1}) \in \cl{B}_{n-1, k} \}$: sequences where only the 10 in a 101 pattern has occurred.
    \item $\cl{D}_{n, k} := \{a \in \cl{I}_n(100, 102, 201) \;|\;  k = \premax(a) < a_n \;\text{and}\;a \;\text{contains}\; 101 \}$: sequences which contain a 101 pattern, but for which the factor $q_3$ is empty.
    \item $\cl{E}_{n, \ell}^{(2)} := \{a \in \cl{I}_n(100, 102, 201) \;|\; \ell = a_n < \premax(a) \;\text{and}\; a \;\text{contains}\; 101 \}$: sequences which contain a 101 pattern and a non-empty factor $q_3$.
    \item $\cl{E}_{n,\ell} = \cl{E}_{n, \ell}^{(1)} \cup \cl{E}_{n, \ell}^{(2)}$.
\end{itemize}

An inversion sequence which avoids $(100, 102, 201)$ but contains $101$ must progress initially from an element of $\cl{A}$ to elements of $\cl{B}$, $\cl{C}$, then $\cl{D}$. It may then also progress to a sequence in $\cl{E}^{(2)}$. This progression is shown in \cref{fig:prog_1176}.

It is important to note here that the sets $\cl{B}$ and $\cl{C}$ are \emph{not disjoint} from the sets $\cl{A}$ and $\cl{E}$. Elements of $\cl{B}_{n, k}$ also occur in $\cl{A}_{n, h}$ for $h > k$, and elements of $\cl{C}_{n, k}$ occur in $\cl{E}^{(1)}_{n, k}$. For a thorough discussion of the purpose of these seemingly redundant objects, see \citet{testart_completing_2024} after Remark 12 (in particular the generating tree diagram on page 13) where they are termed ``phantom objects''. The upshot is that by restricting the future values of inversion sequences, we reduce the number of statistics to track and accordingly simplify the enumeration. All inversion sequences labelled with subscripts $b$ or $c$ in \cref{prop:rule_1176} are excluded from the enumeration of $\cl{I}(>,\le,\ne)$, because their copies with label subscripts $a$ or $e$ \emph{are} enumerated.

% We explain via the following: consider an inversion sequence $a$ occurring in both $\cl{A}$ and $\cl{B}$ with $\premax(a)=k$ and $\max(a)=h$; it is weakly increasing. This inversion sequence may continue to increase, in which case its maximum value and length must be known to allow for all possible ascents. It may instead decrease to a value $j \in [0, h-1]$. If $j<k$ then this inversion sequence may not then return to $h$, since $(k, j, h)$ is a 102 pattern. If $j\ge k$, no such restriction applies. We must know $k$ to determine if the inversion sequence is restricted or not, but additionally tracking $k$ complicates the upcoming succession rule and functional equations. To avoid this, whenever $a$ increases from $k$ to $h$, we make $h-k+1$ copies of $a$; one which we permit only to ascend or to repeat until eventually ascending (landing in $\cl{A}$), and one which we permit only to descend to $k$ or to repeat until eventually descending to $k$ (landing in $\cl{C}$). The latter (an element of $\cl{B}$) no longer needs to track the maximum $h$, since future ascents are restricted exclusively to $h$ and no other value. We use these phantom objects to simplify the succession rule, but ensure the inversion sequences they represent are only actually counted once during enumeration (duplicates do not get counted).

\begin{figure}
\centering
\begin{subfigure}[t]{0.45\textwidth}
    \resizebox{\textwidth}{!}{\input{figures/1_A279562}}
    \caption{A sequence in which the 1 of a 101 pattern has occurred; it is an element of $\cl{B}_{8, 4}$.}
\end{subfigure}
\hspace{4pt}
\begin{subfigure}[t]{0.45\textwidth}
    \resizebox{\textwidth}{!}{\input{figures/2_A279562}}
    \caption{A sequence in which the 10 of a 101 pattern has occurred; it is an element of $\cl{C}_{9, 4}$.}
\end{subfigure}

\vspace{8pt}
\begin{subfigure}[t]{0.45\textwidth}
    \resizebox{\textwidth}{!}{\input{figures/3_A279562}}
    \caption{A sequence containing the pattern 101 which has not yet descended below its second-greatest value; it is an element of $\cl{D}_{13, 4}$.}
\end{subfigure}
\hspace{4pt}
\begin{subfigure}[t]{0.45\textwidth}
    \resizebox{\textwidth}{!}{\input{figures/4_A279562}}
    \caption{A sequence containing the pattern 101 which has descended below its second-greatest value; it is an element of $\cl{E}^{(2)}_{17, 0}$.}
\end{subfigure}
\caption{The progression of a (100, 102, 201)-avoiding inversion sequence growing on the right, as it graduates through the sets $\cl{B}$, then $\cl{C}$, then $\cl{D}$, and landing finally in $\cl{E}^{(2)}$.}
\label{fig:prog_1176}
\end{figure}

\begin{prop}\label{prop:rule_1176}
Give the following labels to the corresponding sequences:
\begin{itemize}
    \item $(n, h)_a$ to each sequence in $\cl{A}_{n, h},$
    \item $(k)_b$ to each sequence in $\cl{B}_{n, k},$
    \item $(k)_{c}$ to each sequence in $\cl{C}_{n, k},$
    \item $(k)_{d}$ to each sequence in $\cl{D}_{n, k},$
    \item $(\ell)_{e}$ to each sequence in $\cl{E}_{n, \ell}.$ 
\end{itemize}
Then a succession rule which generates all inversion sequences in $\cl{I}(100, 102, 201) \equiv \cl{I}(>,\leq,\neq)$ is as follows:

\begin{equation}\Omega_{(>,\leq,\neq)}:\left\{\begin{aligned}
(0, 0)_a \\
(n, h)_a &\leadsto (n+1, i)_a  &&\text{for}\; i \in [h, n] \\
    &\leadsto (i)_b^{n-i} &&\text{for}\; i \in [h, n-1] \\
    &\leadsto (i)_{e} &&\text{for}\; i \in [0, h-1] \\[1ex]
(k)_b &\leadsto (k)_b & \\
    &\leadsto (k)_{c} & \\[1ex]
(k)_{c} &\leadsto (k)_{d} & \\[1ex]
(k)_{d} &\leadsto (k)_{d} & \\
    &\leadsto (i)_{e} &&\text{for}\; i \in [0, k-1] \\[1ex]
(\ell)_{e} &\leadsto (i)_{e} &&\text{for}\; i \in [0, \ell-1].
\end{aligned}\right.
\label{rule:100-102-201}
\end{equation}
\end{prop}

We now obtain the generating function for $\cl{I}(100, 102, 201)$. Let $\alpha_{n,h}$ (resp.\ $\beta_{n,k}$, $\gamma_{n,k}$, $\delta_{n,k}$ and $\veps_{n,\ell}$) be the counting sequences for the sets $\cl{A}_{n,h}$ (resp.\ $\cl{B}_{n,k}$, $\cl{C}_{n,k}$, $\cl{D}_{n,k}$ and $\cl{E}_{n,\ell}$), and then define the generating functions $A(z,x)$ (resp.\ $B(z,x)$, $C(z,x)$, $D(z,x)$ and $E(z,x)$) with $z$ marking the length $n$ and $x$ marking the second parameter $h,k$ or $\ell$ appropriately.
% \begin{align*}
% A(z, x) &= \sum_{n, h \geq 0} \alpha_{n, h} z^n x^h \\
% B(z, x) &= \sum_{n, k \geq 0} \beta_{n, k} z^n x^k \\
% C(z, x) &= \sum_{n, k \geq 0} \gamma_{n, k} z^n x^k \\
% D(z, x) &= \sum_{n, k \geq 0} \delta_{n, k} z^n x^k \\
% E(z, x) &= \sum_{n, \ell \geq 0} \veps_{n, \ell} z^n x^\ell.
% \end{align*}
\begin{lem}
\label{eqs100-102-201}
The generating functions $A, B, C, D, E$ and $F(z) = \sum_{n=0}^\infty |\cl{I}_n(>, \le, \ne)|z^n$ satisfy the following equations:
\begin{align}
A(z, x) &= 1 + \frac{z}{1-x}\left(A(z, x) - xA(zx, 1)\right) \\
B(z, x) &= zB(z, x) + \frac{z}{1-x}\left(z\parz A(z, x) - x\parx A(z, x) + \frac{x}{1-x}\Big(A(zx, 1) - A(z, x)\Big)\right) \\
C(z, x) &= zB(z, x) \\
D(z, x) &= zD(z, x) + zC(z, x) \\
E(z, x) &= \frac{z}{1-x}\Big(E(z, 1) - E(z, x) + A(z, 1) - A(z, x) + D(z, 1) - D(z, x)\Big) \\
F(z) &= A(z, 1) + D(z, 1) + E(z, 1).
\end{align}
\end{lem}

\begin{proof} 
We convert the succession rule $\Omega_{(>,\leq,\neq)}$ into functional equations:
\begin{align*}
    A(z, x) &= 1 + \sum_{n, h \ge 0}\alpha_{n, h} z^{n+1} \sum_{i=h}^n x^i \\
    &= 1 + \frac{z}{1-x}\left(A(z, x) - xA(zx, 1)\right) \\[2ex]
% \end{align*}
% \begin{align*}
    B(z, x) &= \left(\sum_{n, k \geq 0}\beta_{n, k}z^{n+1}x^k \right) + \left(\sum_{n, h \geq 0}\alpha_{n, h}z^{n+1}\sum_{i=h}^{n-1}(n-i)x^i \right) \\
    &= z\left(\sum_{n, k \geq 0}\beta_{n, k} z^n x^k \right) + z\left(\sum_{n, h \geq 0}\alpha_{n, h}z^n\frac{1}{1-x}\left(nx^h - hx^h + \frac{x}{1-x}(x^n - x^h) \right) \right) \notag\\
    &= zB(z, x) + \frac{z}{1-x}\left(z\parz A(z, x) - x\parx A(z, x) + \frac{x}{1-x}\Big(A(zx, 1) - A(z, x) \Big) \right).\\[2ex]
% \end{align*}
% \begin{align*}
    C(z, x) &= \sum_{n, k \geq 0} \beta_{n, k} z^{n+1} x^k \\
    &= zB(z, x).\\[2ex]
% \end{align*}
% \begin{align*}
    D(z, x) &= \left(\sum_{n, k \geq 0} \delta_{n, k} z^{n+1} x^k \right) + \left(\sum_{n, k \geq 0} \gamma_{n, k} z^{n+1} x^k \right)\\
    &= zD(z, x) + zC(z, x).\\[2ex]
% \end{align*}
% \begin{align*}
    E(z, x) &= \left(\sum_{n, \ell \geq 0}\veps_{n, \ell}z^{n+1}\sum_{i=0}^{\ell-1}x^i \right) + \left(\sum_{n, h \geq 0}\alpha_{n, h}z^{n+1}\sum_{i=0}^{h-1}x^i \right) + \left(\sum_{n, k \geq 0}\delta_{n, k}z^{n+1}\sum_{i=0}^{k-1}x^i \right) \\
    &= \frac{z}{1-x}\Big(E(z, 1) - E(z, x) + A(z, 1) - A(z, x) + D(z, 1) - D(z, x)\Big).
\end{align*}
The expression $F(z) = A(z, 1) + D(z, 1) + E(z, 1)$ is a direct consequence of the construction: the inversion sequences of interest either avoid the pattern 101 (and so are included in $\cl{A}$ or $\cl{E}^{(1)}$) or they contain 101 (and are included in $\cl{D}$ or $\cl{E}^{(2)}$).
\end{proof}

\begin{thm}
\label{gf100,102,201}
The generating function for inversion sequences avoiding the set of patterns \\$(100, 102, 201) \equiv (>, \leq, \neq)$ is given by
\begin{align}
    F(z) &= \frac{2 + z - 10z^2 + 4z^3 - (2 - 3z)\sqrt{1 - 4z - 4z^2}}{8z(1-z)^2} \\
    &= 1 + z + 2 z^2 + 6 z^3 + 21 z^4 + 78 z^5 + 299 z^6 + 1176 z^7 + 
 4729 z^8 + 19378 z^9 + 80667 z^{10} + \dots \notag
\end{align}
\end{thm}
\begin{proof} We solve the functional equations for $A, D,$ and $E$ to obtain $F(z) = A(z, 1) + D(z, 1) + E(z, 1)$. The explicit form of $A(z, x)$ is derived by \citet{testart_completing_2024}:
\begin{align}\label{eqn:soln_A_1176}
    A(z, x) = \frac{1 - 2x + \sqrt{1 - 4zx}}{2(1 - x - z)}.
\end{align}
We find $B(z, x)$, $C(z, x)$ and $D(z, x)$ using \textsc{Mathematica} to assist with the computations:
\begin{align}
    B(z, x) &= \frac{1 - x(1 - 2zx + 2z^2) - 2z(1 - z) - (1 - x)(1 - 2z)\sqrt{1 - 4zx}}{2(1-x)(1-z)(1-x-z)^2} \label{eqn:soln_B_1176}\\
    C(z, x) &= \frac{z(1 - x(1 - 2zx + 2z^2) - 2z(1 - z) - (1 - x)(1 - 2z)\sqrt{1 - 4zx})}{2(1-x)(1-z)(1-x-z)^2} \label{eqn:soln_C_1176}\\
    D(z, x) &= \frac{z^2(1 - x(1 - 2zx + 2z^2) - 2z(1 - z) - (1 - x)(1 - 2z)\sqrt{1 - 4zx})}{2(1-x)(1-z)^2(1-x-z)^2}.\label{eqn:soln_D_1176}
\end{align}
Obtaining an expression for $E(z, 1)$ requires the kernel method. From \cref{eqs100-102-201}:
\[E(z, x) = \frac{z}{1-x}\Big(E(z, 1) - E(z, x) + A(z, 1) - A(z, x) + D(z, 1) - D(z, x)\Big)\]
which becomes
\begin{align}
    (1 - x + z)E(z, x) = z\Big(E(z, 1) + A(z, 1) - A(z, x) + D(z, 1) - D(z, x)\Big).
\end{align}
We set $x = 1 + z$, and the right-hand side gives
\begin{align}
    0 = E(z, 1) + A(z, 1) - A(z, 1 + z) + D(z, 1) - D(z, 1 + z).
\end{align}
Solving for $E(z, 1)$ gives
\begin{align}
    E(z, 1) = \frac{-2 + 5z + 2z^2 - 4z^3 + 4(1 - z - z^2)\sqrt{1 - 4z} - (2 - 3z)\sqrt{1 - 4z - 4z^2}}{8z(1 - z)^2}.
\end{align}

Finally we obtain the stated expression for $F(z)$ by summing $A(z, 1), D(z, 1)$, and $E(z, 1)$.
\end{proof}

% \subsection{A279563, (102, 201, 210)}
\subsection{Class 1253: $(>, \neq, \neq) \equiv (102, 201, 210)$}
\label{ssec:1253}

The triple of relations $(>, \ne, \ne)$ corresponds to the set of patterns  $(102, 201, 210)$. We construct a succession rule growing on the right which is based on the succession rule $\Omega_{(>, \le, \ne)}$ in \eqref{rule:100-102-201}. From this rule we extract the corresponding functional equations and solve them to obtain the generating function for $\cl{I}(>, \ne, \ne)$.

Let $a \in \cl{I}(102, 201)$ be arbitrary, and recall the factorisation $a = p_1 \circ p_2 \circ p_3$ from \cref{ssec:1176}; we will modify it to describe inversion sequences which additionally avoid 210. If $p_2$ contains $\premax(a)$, then $p_3$ is required to be empty. If $p_2$ does not contain $\premax(a)$ (i.e.\ $\premax(a)$ occurs exclusively in $p_1$), then $p_3$ may contain any one value. The updated characterisation is as follows:

\begin{prop}
    An inversion sequence $a \in \cl{I}_n$ avoids the set of patterns $(102, 201, 210)$ if and only if it can be factored as either $p_1 \circ p_2$ or $p_1 \circ r_2$, where:
\begin{itemize}
    \item $p_1$ and $p_2$ are as in \cref{p1p2p3_102-201},
    % \item $p_2$ is a sequence of terms from $\{\max(a),\; \premax(a)\}$ beginning on the term $\max(a)$,
    \item $r_2 = (\max(a), \ldots, \max(a), k, \ldots, k)$, where $0 \leq k < \premax(a)$, and $\max(a)$ and $k$ each occur at least once.
\end{itemize}
\end{prop}

As before, we will first consider $(102, 201, 210)$-avoiding inversion sequences which also avoid $101$, then afterwards consider those which contain $101$. The pattern $101$ may occur only in the factor $p_2$. Accordingly, the $101$-avoiding inversion sequences will be unimodal with no more than one distinct value occurring after the maximum.

Recall the definition for $\cl{A}_{n, h}$ given in \cref{ssec:1176}. 
% along with $\cl{A} = \bigcup_{n, h \geq 0} \cl{A}_{n, h}$ and $\alpha_{n, h} = |\cl{A}_{n, h}|$.
Define the following set:
\begin{itemize}
    \item $\cl{E}_{n, \ell}^\dg = \{a \in \cl{I}_n(101, 102, 201, 210) \;|\; a_n = \ell \;\text{and}\; a_n < \max(a)\}$: the set of unimodal inversion sequences of length $n$ such that $\ell < \max(a)$ is the only distinct value which occurs after the final $\max(a)$.
\end{itemize}
%Let $\veps_{n, \ell}^{\dg} = |\cl{E}_{n, \ell}^\dg|$ and $\cl{E}^\dg = \bigcup_{n, \ell \geq 0} \cl{E}_{n, \ell}^\dg$.

We can then write a succession rule for inversion sequences avoiding the set of patterns \newline $(101, 102, 201, 210)$, where an element $a \in \cl{E}^{\dg}_{n, \ell}$ may progress only to the inversion sequence $a \circ \ell$.

\begin{prop}
Give the label $(n, h)_a$ to each sequence in $\cl{A}_{n, h}$, and give the label $(\ell)_{e^\dg}$ to each sequence in $\cl{E}_{n, \ell}^\dg$. Then we have the following succession rule for $\cl{I}_n(101, 102, 201, 210)$:

\begin{equation}
\Omega_{(101, 102, 201, 210)}: \left\{ \begin{aligned}
(0, 0)_a \\
(n, h)_a &\leadsto (n+1, i)_a &&\text{for}\; i \in [h, n] \\
        &\leadsto (i)_{e^\dg} &&\text{for}\; i \in [0, h-1] \\
(\ell)_{e^\dg} &\leadsto (\ell)_{e^\dg}.
\end{aligned} \right.
\end{equation}
\end{prop}

We subsequently progress to counting inversion sequences in $\cl{I}_n(102, 201, 210)$ which \emph{contain} the pattern 101. As noted above, a sequence $a$ will only contain the pattern 101 if it can be factored as $a = p_1 \circ p_2$ and $p_2$ contains the pattern 101. We already have succession rules for some of these sequences from \cref{prop:rule_1176}: those elements of $\cl{B}$ in \cref{ssec:1176}. We must modify the sets $\cl{C}$ and $\cl{D}$ from \cref{ssec:1176} to allow repeated instances of the premaximum, since 100-avoidance is no longer necessary. Define
\begin{itemize}
    \item $\cl{C}_{n, k}^\dg := \{(a)_{n-i} \circ (k)^i\; |\; i \in [1,\; n-2]\; \text{and}\; (a)_{n-i} \in \cl{B}_{n-i, k} \}$: sequences where only the 10 in a 101 pattern has occurred. As such, they consist of a sequence containing only a 1 of the 101 pattern, concatenated with a string of $k$ to form the 0 of the 101 pattern.
    \item $\cl{D}_{n, k}^\dg := \{(a)_n \in \cl{I}_n(102, 201)\; |\; k = \premax(a) \leq a_n\; \text{and}\; (a)_n\;\text{contains}\; 101\}$: sequences which contain a 101 pattern, but for which the ``descending'' factor $p_3$ from the characterisation in \cref{p1p2p3_102-201} is empty.
\end{itemize}

We have no need of the set $\cl{E}^{(2)}$, since inversion sequences in $\cl{I}_n(102, 201, 210)$ which contain 101 do not have a $p_3$ factor. It was this factor which produced sequences in $\cl{E}^{(2)}$.

\begin{rem}
\label{rem_102,201,210_w/101}
Explicitly, an inversion sequence $a \in \cl{I}_n(102, 201, 210)$ contains the pattern 101 if and only if $a \in \cl{D}^{\dagger}_{n, k}$ for some $k \geq 0$.
\end{rem}

The successions involving elements of $\cl{C}^\dagger$ and $\cl{D}^\dagger$ are the same as those for $\cl{C}$ and $\cl{D}$ respectively in \cref{prop:rule_1176}, with the addition of $\cl{C}_{n, k}^\dg \to \cl{C}_{n+1, k}^\dg$ and a second instance of $\cl{D}_{n, k}^\dg \to \cl{D}_{n+1, k}^\dg$, both of which correspond to appending the premaximum of the sequence. Modifying the rule $\Omega_{(102, 201)}$ from \citet{testart_completing_2024} to exclude the successions $\cl{D} \to \cl{E}$ and $\cl{E}_{n, \ell} \to \cl{E}_{n+1, i}$ for $i < \ell$, we may write our full succession rule.

\begin{prop}\label{prop:rule_1253}
Give the following labels to the corresponding sequences:
\begin{itemize}
    % \item $(n, h)_a$ to each sequence in $\cl{A}_{n, h}$,
    % \item $(k)_b$ to each sequence in $\cl{B}_{n, k}$,
    \item $(k)_c^\dg$ to each sequence in $\cl{C}_{n, k}^\dg$,
    \item $(k)_d^\dg$ to each sequence in $\cl{D}_{n, k}^\dg$,
    \item $(\ell)_e^\dg$ to each sequence in $\cl{E}_{n, \ell}^\dg$,
\end{itemize}
with labels $(n,h)_a$ and $(k)_b$ as defined in \cref{prop:rule_1176}. Then a succession rule which generates all inversion sequences in $ \cl{I}(>, \ne, \ne) \equiv \cl{I}(102, 201, 210)$ is as follows:
\begin{equation}\Omega_{(>,\neq,\neq)}:\left\{\begin{aligned}
(0, 0)_a \\
(n, h)_a &\leadsto (n+1, i)_a && \text{for}\; i \in [h, n] \\
    &\leadsto (i)_b^{n-i} && \text{for}\; i \in [h, n-1] \\
    &\leadsto (i)_{e^\dg} && \text{for}\; i \in [0, h-1] \\[1ex]
(k)_b &\leadsto (k)_b \\
    &\leadsto (k)_{c^\dg} \\[1ex]
(k)_{c^\dg} &\leadsto (k)_{c^\dg} \\
    &\leadsto (k)_{d^\dg} \\[1ex]
(k)_{d^\dg} &\leadsto (k)_{d^\dg}^2 \\[1ex]
(\ell)_{e^\dg} &\leadsto (\ell)_{e^\dg}. \\
\end{aligned}\right.\end{equation}
\end{prop}

As with \cref{ssec:1176} the tree generated by $\Omega_{(>, \ne, \ne)}$ contains \emph{phantom objects}, so that the number of labels at depth $n$ exceeds the number of $(>, \ne, \ne)$-avoiding inversion sequences. The phantom objects correspond to labels with subscripts $b$ and $c^\dg$.\footnote{See again Remark 13 and the diagram on page 13 of \citet{testart_completing_2024} for a discussion of phantom objects.}

We now obtain the generating function for $\cl{I}(>, \ne, \ne)$. Let %$\gamma_{n,k}^\dg$ (resp.\ $\delta_{n,k}^\dg$ and $\veps_{n,\ell}^\dg$) be the counting sequences for the sets $\cl{C}_{n,k}^\dg$ (resp.\ $\cl{D}_{n,k}^\dg$ and $\cl{E}_{n,\ell}^\dg$),
$C^\dg(z,x)$ (resp. $D^\dg(z,x)$ and $E^\dg(z,x)$) be the generating functions for the classes $\cl{C}^\dg$ (resp. $\cl{D}^\dg$ and $\cl{E}^\dg$) with $z$ marking the length $n$ and $x$ marking the second parameter $k$ or $\ell$ appropriately. Recall the functions $A(z, x)$ and $B(z, x)$ from \cref{ssec:1176}.

\begin{lem}
\label{lem:eqs102-201-210}
    The bivariate generating functions $C^\dg, D^\dg, E^\dg$ and ordinary generating function \newline
    $F^\dg(z) = \sum_{n=0}^\infty |\cl{I}_n(102, 201, 210)|z^n$ satisfy the following equations:
    \begin{align}
    C^\dg(z, x) &= zC^\dg(z, x) + zB(z, x) \\
    D^\dg(z, x) &= 2zD^\dg(z, x) + zC^\dg(z, x) \\
    E^\dg(z, x) &= zE^\dg(z, x) + \frac{z}{1-x}\Big(A(z, 1) - A(z, x)\Big) \label{eqn:1253_E_eqn}\\
    F^\dg(z) &= A(z, 1) + D^\dg(z, 1) + E^\dg(z, 1).
    \end{align}
\end{lem}
We omit the proof as it follows directly from \cref{prop:rule_1253}, along the lines of the proof for \cref{eqs100-102-201}.
% \begin{proof} We convert the succession rule into functional equations:
% \begin{align}
% C^\dg(z, x) &= \left(\sum_{n, k \geq 0}\gamma_{n, k}^\dg z^{n+1} x^k \right) + \left(\sum_{n, k \geq 0} \beta_{n, k} z^{n+1} x^h \right) \\
% &= zC^\dg(z, x) + zB(z, x). \\[2ex]
% D^\dg(z, x) &= \left(\sum_{n, k \geq 0} 2\delta_{n, k}^\dg z^{n+1} x^h \right) + \left(\sum_{n, k \geq 0} \gamma_{n, k}^\dg z^{n+1} x^h \right) \\
% &= 2zD^\dg(z, x) + zC^\dg(z, x). \\[2ex]
% E^\dg(z, x) &= \left(\sum_{n, \ell \geq 0}\veps_{n, \ell}^\dg z^{n+1}x^\ell \right) + \left(\sum_{n, h \geq 0}\alpha_{n, h}z^{n+1}\sum_{i=0}^{h-1}x^i \right) \\
% &= z\left(\sum_{n, \ell \geq 0}\veps_{n, \ell}^\dg z^n x^\ell \right) + z\left(\sum_{n, h \geq 0}\alpha_{n, h}z^n \frac{1 - x^h}{1-x} \right) \notag\\
% &= zE^\dg(z, x) + \frac{z}{1-x}\Big(A(z, 1) - A(z, x)\Big). 
% \end{align}
% The equation $F(z) = A(z, 1) + D(z, 1) + E^\dg(z, 1)$ is a direct consequence of the construction, where every element of $\cl{C}^\dg$ and $\cl{D}^\dg$ also occurs in $\cl{A}$ or $\cl{E}^\dg$ and does not need to be counted multiple times.
% \end{proof}

% \subsection{A279560, (100, 102, 201, 210)}

\begin{thm}
\label{gf102,201,210}
The generating function for inversion sequences avoiding the set of patterns \\$(102, 201, 210) \equiv (>, \neq, \neq)$ is given by
\begin{align}
F^\dg(z) &= \frac{2 - 15z + 32z^2 - 16z^3 + z(1 - 2z)(1+2z)\sqrt{1 - 4z}}{2(1-z)^2(1-2z)(1-4z)} \\
&= 1 + z + 2 z^2 + 6 z^3 + 22 z^4 + 85 z^5 + 328 z^6 + 1253 z^7 + 
 4754 z^8 + 17994 z^9 + 68158 z^{10} + \dots \notag
\end{align}
\end{thm}

\begin{proof}
We will solve the equations in \cref{lem:eqs102-201-210} for $D^\dg(z, 1)$ and $E^\dg(z, 1)$, and subsequently obtain the generating function $F^\dg(z) = A(z, 1) + D^\dg(z, 1) + E^\dg(z, 1)$.

$A(z, x)$ and $B(z, x)$ are as solved in \cref{ssec:1176}, so we have
\begin{align}
C^\dg(z, x) &= \frac{z(1 - x(1 - 2zx + 2z^2) - 2z(1 - z) - (1 - x)(1 - 2z)\sqrt{1 - 4zx})}{2(1-x)(1-z)^2(1-x-z)^2}, \\
D^\dg(z, x) &= \frac{z^2(1 - x(1 - 2zx + 2z^2) - 2z(1 - z) - (1 - x)(1 - 2z)\sqrt{1 - 4zx})}{2(1-x)(1-z)^2(1-2z)(1-x-z)^2}.
\end{align}
% From \cref{eqs102-201-210} we have
% \[E^\dg (z, x) = zE^\dg (z, x) + \frac{z}{1-x}\Big(A(z, 1) - A(z, x)\Big),\]
% which yields
Then substituting into \eqref{eqn:1253_E_eqn} gives
\begin{equation}
E^\dg(z, x) = \frac{1 - x + 2zx - 2z - (1 - x - z)\sqrt{1 - 4z} - z\sqrt{1 - 4zx}}{2(1-x)(1-z)(1-x-z)}.
\end{equation}
Thus, we obtain the stated expression for $F^\dg(z)$ by summing $A(z, 1) + D^\dg(z, 1) + E^\dg(z, 1)$ (taking limits $x\to1$ as required).
\end{proof}

\subsection{Class 1016: $(>,-,\neq) \equiv (100, 102, 201, 210)$}
\label{ssec:1016}

% \nick{[Nathan to fix up this section, some of it is not needed as it is solved in previous sections]}\nathan{[All done, subject to a final look-over once all sections are complete]}

An explicit expression for $I_n(>,-,\neq)$ is given by \citet[Sec.\ 3.1.2]{martinez_patterns_2018}, in the form of a four-fold sum. The authors give no derivation. Here we will give a succession rule and use it to solve the generating function.

The triple of relations $(>, -, \ne)$ corresponds to the set of patterns $(100, 102, 201, 210)$. We construct a succession rule growing on the right which is based on the succession rule $\Omega_{(>, \le, \ne)}$ in \eqref{rule:100-102-201}. From this rule we extract the corresponding functional equations and solve them to obtain the generating function for $\cl{I}(>, -, \ne)$.

Define the following sets of objects: 
\begin{itemize}
    \item $\cl{E}^*_{n, \ell} = \{a \circ \ell\;|\; a \in \cl{I}_{n-1}(10)\; \text{and}\; \ell < a_{n-1}\}$: the set of unimodal inversion sequences of length $n$ which consist of a non-decreasing inversion sequence followed by a single descent to the value $\ell$.
\end{itemize}
The set $\mcI(100, 102, 201, 210)$ is the intersection of the sets $\mcI(100, 102, 201)$ and $\mcI(102, 201, 210)$, so we easily obtain the succession rule for $\cl{I}(>, -, \ne)$ from the corresponding intersection of $\Omega_{(>, \le, \ne)}$ and $\Omega_{(>, \ne, \ne)}$.\footnote{Explicitly, the `intersection' of succession rules consists of all successions present in both rules simultaneously. The set $\cl{E}^*$ is the intersection of $\cl{E}$ and $\cl{E}^\dg$. In fact, the objects of $\cl{E}^*$ cannot progress to any other inversion sequences while avoiding the necessary patterns, so the label is immaterial.}

\begin{prop}\label{prop:rule_1016}
Assign the following labels to the corresponding inversion sequences:
\begin{itemize}
    \item $(\ell)_{e^*}$ to each sequence in $\cl{E}^*_{n, \ell}$.
\end{itemize}
Then a succession rule which generates inversion sequences in $\cl{I}(>, -, \ne)$ is as follows:
\begin{equation}
\Omega_{(>, -, \ne)}:\left\{\begin{aligned}
(0, 0)_a \\
(n, h)_a &\leadsto (n+1, i)_a  &&\text{for}\; i \in [h, n] \\
    &\leadsto (i)_b^{n-i} &&\text{for}\; i \in [h, n-1] \\
    &\leadsto (i)_{e^*} &&\text{for}\; i \in [0, h-1] \\[1ex]
(k)_b &\leadsto (k)_b & \\
    &\leadsto (k)_c& \\[1ex]
(k)_c &\leadsto (k)_d & \\[1ex]
(k)_{d} &\leadsto (k)_d.
\end{aligned}\right.
\label{eqn:rule100-102-201-210}
\end{equation}
\end{prop}
%It's modified from the succession rule $\Omega_{(100, 102, 201)}$ by insisting the $(\_)_e$ objects have no children, and that $(\_)_d$ objects have no $(\_)_e$ objects as children. Set $F(z) = \sum_{n\geq0}|\cl{I}_n(100, 102, 201, 210)|z^n$.
The tree generated by $\Omega_{(>, -, \ne)}$ is not isomorphic to a combinatorial generating tree for $\cl{I}(>, -, \ne)$, since the objects in $\cl{B}_{n, k}$ appear in $\cl{A}_{n, h}$ for some $h > k$, and the objects in $\cl{C}_{n, k}$ appear in $\cl{E}^*_{n, k}$.

% \subsubsection{Equations and Generating Function}
% We continue to use $\alpha_{n, h} = |\cl{A}_{n, h}|$ and $\beta_{n, k} = |\cl{B}_{n, k}|$, and similarly define $\gamma^*_{n, k}$, $\delta^*_{n, k}$ and $\veps^*_{n, \ell}$.
We reuse the generating functions $A(z,x), B(z,x), C(z, x)$ and $D(z,x)$ from \cref{ssec:1176} and set $E^*(z,x)$ to be the generating function of the class $\cl{E}^*$, with $z$ associated to length $n$ and $x$ associated to final value $\ell$. Finally we define
% \begin{align*}
% A(z, x) &= \sum_{n, h \geq 0} \alpha_{n, h} z^n x^h \\
% B(z, x) &= \sum_{n, k \geq 0} \beta_{n, k} z^n x^k \\
% C(z, x) &= \sum_{n, k \geq 0} \gamma_{n, k} z^n x^k \\
% D(z, x) &= \sum_{n, k \geq 0} \delta_{n, k} z^n x^k \\
% E(z, x) &= \sum_{n, \ell \geq 0} \veps_{n, \ell} z^n x^\ell,
% \end{align*}
% and let
\begin{equation*}
    F^*(z) = \sum_{n \ge 0}|\cl{I}_n(>, -, \ne)|z^n.
\end{equation*}

\begin{lem}
The bivariate generating functions $A$, $B$, $C$, $D$, $E^*$ and the ordinary generating function $F^*$ satisfy the following equations:
\begin{align}
    A(z, x) &= 1 + \frac{z}{1-x}\Big(A(z, x) -  xA(zx, 1)\Big) \\
    B(z, x) &= zB(z, x) + \frac{z}{1-x}\left(z\parz A(z, x) - x\parx A(z, x) + \frac{x}{1-x}\Big(A(zx, 1) - A(z, x) \Big) \right), \\
    C(z, x) &= zB(z, x) \\
    D(z, x) &= zD(z, x) + zC(z, x), \\
    E^*(z, x) &= \frac{z}{1-x}\Big(A(z, 1) - A(z, x)\Big), \\
    F^*(z) &= A(z, 1) + D(z, 1) + E^*(z, 1).
\end{align}
\end{lem}
The proof just follows from \cref{prop:rule_1016}. We omit the details as the proof is very similar to that of \cref{eqs100-102-201}.

% \begin{proof}
% \nick{[can probably skip this proof, it looks almost identical to the previous one... indeed much of it is completely identical]}

% We convert the succession rule $\Omega_{(100, 102, 201, 210)}$ into functional equations:
% \begin{align*}
% A(z, x) &= 1 + \sum_{n,h\geq0}\alpha_{n,h}z^{n+1}\sum_{i=h}^n x^i \\
% &= 1 + \frac{z}{1-x}\Big(A(z, x) - xA(zx, 1)\Big),
% \end{align*}
% \begin{align*}
% B(z, x) &= \left(\sum_{n,h\geq0} \alpha_{n,h} z^{n+1}\sum_{i=h}^{n-1}(n-i)x^i\right) + \left(\sum_{n,k \geq0}\beta_{n,k}z^{n+1}x^k\right) \\
% &= zB(z, x) + \frac{z}{1-x}\left(z\parz A(z, x) - x\parx A(z, x) + \frac{x}{1-x}\Big(A(zx, 1) - A(z, x) \Big) \right),
% \end{align*}
% \begin{align*}
% C(z, x) &= \sum_{n, k \geq 0} \beta_{n, k} z^{n+1} x^h \\
% &= zB(z, x),
% \end{align*}
% \begin{align*}
% D(z, x) &= \left(\sum_{n, k \geq 0} \delta_{n, k} z^{n+1} x^h \right) + \left(\sum_{n, k \geq 0} \gamma_{n, k} z^{n+1} x^h \right)\\
% &= zD(z, x) + zC(z, x),
% \end{align*}
% \begin{align*}
% E(z, x) &= \sum_{n,h\geq0}\alpha_{n,h}z^{n+1}\sum_{i=0}^{h-1}x^i \\
% &= \frac{z}{1-x}\Big(A(z, 1) - A(z, x)\Big),
% \end{align*}
% \begin{align*}
% F(z) = A(z, 1) + D(z, 1) + E(z ,1).
% \end{align*}
% \end{proof}

\begin{thm}
    The generating function for inversion sequences avoiding the set of patterns \newline $(100, 102, 201, 210) \equiv (>, -, \ne)$ is given by
    \begin{align}
        F^*(z) &= \frac{(1-4z)(1-2z)(3-2z)- (1 - 8z + 12z^2 - 2z^3)\sqrt{1 - 4z}}{2(1 - z)^2(1 - 4z)} \\
        &= 1 + z + 2 z^2 + 6 z^3 + 21 z^4 + 76 z^5 + 277 z^6 + 1016 z^7 + 
 3756 z^8 + 13998 z^9 + 52554 z^{10} + \dots \notag
    \end{align}
\end{thm}

\begin{proof}
The solutions to $A, B, C,$ and $D$ are as given in \eqref{eqn:soln_A_1176}, \eqref{eqn:soln_B_1176}, \eqref{eqn:soln_C_1176}, and \eqref{eqn:soln_D_1176}, respectively.
% Use kernel method:
% \begin{align}
%     A(z, x) &= \frac{1-x - zxA(zx, 1)}{1-x-z},
% \end{align}
% so
% \begin{align}
%     &&1 - (1-z) - z(1-z)A(z-z^2, 1) = 0 \\
%     &\Rightarrow & A(z-z^2, 1) = \frac{1}{1-z} \\
%     &\Rightarrow & A(z, 1) = \frac{1-\sqrt{1-4z}}{2z}
% \end{align}
% and so
% \begin{align}
%     A(z, x) = \frac{1-2x+\sqrt{1-4zx}}{2(1-x-z)}.
% \end{align}
% Similarly, we solved for $B(z, x)$ when working on $\cl{I}(102, 201)$:
% \begin{align}
%     B(z, x) = \frac{1 - x(1 - 2zx + 2z^2) - 2z(1 - z) - (1 - x)(1 - 2z)\sqrt{1 - 4zx}}{2(1-x)(1-z)(1-x-z)^2}
% \end{align}
% which immediately produces
%\begin{align}
%    C^*(z, x) = \frac{z(1 - x(1 - 2zx + 2z^2) - 2z(1 - z) - (1 - x)(1 - 2z)\sqrt{1 - 4zx})}{2(1-x)(1-z)(1-x-z)^2}
%\end{align}
%and
%\begin{align}
%    D^*(z, x) = \frac{z^2(1 - x(1 - 2zx + 2z^2) - 2z(1 - z) - (1 - x)(1 - 2z)\sqrt{1 - 4zx})}{2(1-x)(1-z)^2(1-x-z)^2}.
%\end{align}
%Then (taking the limit $x \to 1$):
%\begin{align}
%    D^*(z, 1) = \frac{1 - 4z - 2 z^2 - (1 - 2 z)\sqrt{1 - 4 z}}{2(1 - z)^2}.
%\end{align}

From $E^*(z, x) = \frac{z}{1-x}\Big(A(z, 1) - A(z, x)\Big)$ we obtain
\begin{align}
    E^*(z, x) %&= \frac{z}{1-x}\left(\frac{1-\sqrt{1-4z}}{2z} - \frac{1-2x+\sqrt{1-4zx}}{2(1-x-z)}\right) \\
    % &= \frac{z}{1-x}\frac{(1-x-z)-(1-x-z)\sqrt{1-4z} - z + 2zx - z\sqrt{1-4zx}}{2z(1-x-z)} \\
    &= \frac{(1-x)(1-2z)-(1-x-z)\sqrt{1-4z} - z\sqrt{1-4zx}}{2(1-x)(1-x-z)}.
\end{align}
Then (taking the limit $x \to 1$):
\begin{align}
    E^*(z, 1) %&= \frac{-1 + \sqrt{1 - 4 z} + 2 z (1 + \frac{z}{\sqrt{1 - 4 z}})}{2z} \\
    % &= \frac{-1(1-4z) + (1-4z)\sqrt{1-4z}+2z(1-4z) + 2z^2\sqrt{1-4z}}{2z(1-4z)} \\
    &= \frac{-1+6z-8z^2 + (1-4z+2z^2)\sqrt{1-4z}}{2z(1-4z)}.
\end{align}

Combining these, we obtain
\begin{align}
    F^*(z) &= A(z, 1) + D(z, 1) + E^*(z, 1) \\
    % &= \frac{3 - 20z + 36z^2 - 
    % 16z^3 - (1 - 8z + 12z^2 - 2z^3)\sqrt{1 - 4z}}{2(1 - z)^2(1 - 4z)} \\
    &= \frac{(1-4z)(1-2z)(3-2z)- (1 - 8z + 12z^2 - 2z^3)\sqrt{1 - 4z}}{2(1 - z)^2(1 - 4z)} \\
    &= 1 + z + 2 z^2 + 6 z^3 + 21 z^4 + 76 z^5 + 277 z^6 + 1016 z^7 + 
 3756 z^8 + 13998 z^9 + \ldots.
\end{align}
\end{proof}

\subsection{Class 830: $(\ne, >, \ge) \equiv (010, 120, 210)$}
\label{ssec:830}
In this section we present a succession rule grown on the right for the class $\cl{I}(\ne, >, \ge)$, but we do not believe this class to have an algebraic generating function. The construction tracks the length, maximum value, and premaximum value of each inversion sequence. We make the observation that inversion sequences which avoid both 120 and 210 must end on either their maximum value or premaximum value.

Define the sets
\begin{itemize}
    \item $\cl{S}_{n, h, k}$ to be inversion sequences in $\cl{I}_n(\ne, >, \ge)$ ending on their maximum $h$ and with premaximum $k$. For the sake of this definition, set $k=0$ for sequences where $h=0$ (which consist exclusively of zeros).
    \item $\cl{T}_{n, h, k}$ to be inversion sequences in $\cl{I}_n(\ne, >, \ge)$ ending on their premaximum value $k$ with maximum $h \ge 2$.
\end{itemize}

Elements of both $\cl{S}_{n, h, k}$ and $\cl{T}_{n, h, k}$ may repeat their final value which increments $n$ to $n+1$, or may ascend to a new maximum $i > h$ which results in an element of $\cl{S}_{n+1, i, h}$ where the former maximum is now the premaximum.

Elements of $\cl{S}_{n, h, k}$ may descend to a new premaximum $i \in [k+1, h-1]$ (so long as $h > k+1$). They must not return to the current premaximum $k$, which would produce the 010 pattern $(k, h, k)$. Elements of $\cl{T}_{n,h,k}$ may ascend to a new premaximum $i \in [k+1, h-1]$ (so long as $h > k+1$). They may also return to the current maximum $h$.

Assign the labels $(n, h, k)_s$ to elements of $\cl{S}_{n, h, k}$ and $(n, h, k)_t$ to elements of $\cl{T}_{n, h, k}$. The following is a succession rule for the class $\cl{I}(\ne, >, \ge)$:

\begin{equation}\Omega_{(\ne, >, \ge)}:\left\{\begin{aligned}
(0, 0, 0)_s \\
(n, h, k)_s &\leadsto (n+1, h, k)_s \\
&\leadsto (n+1, i, h)_s && \text{for}\; i \in [h+1, n] \\
&\leadsto (n+1, h, i)_t && \text{for}\; i \in [k+1, h-1], \;\text{if}\; h>0\\[1ex]
(n, h, k)_t &\leadsto (n+1, h, k)_s \\
&\leadsto (n+1, i, h)_s &&\text{for}\; i \in [h+1, n] \\
&\leadsto (n+1, h, i)_t &&\text{for}\; i \in [k, h-1].
\end{aligned}\right.\end{equation}

Let $S(z, x, y)$ and $T(z, x, y)$ be the generating functions for the classes $\cl{S}$ and $\cl{T}$ respectively, where $z$ is assigned to length $n$, $x$ to maximum $h$, and $y$ to premaximum $k$. Let $F^\ddagger(z) = \sum_{n\geq0}|\cl{I}_n(\ne, >, \ge)|z^n$. We obtain the following functional equations from the rule $\Omega_{(\ne, >, \ge)}$:
\begin{align}
    S(z, x, y)
    &= 1 + zS(z, x, y) + zT(z, x, y) \notag\\
    & + \frac{zx}{1-x}\Big(S(z, xy, 1) - S(zx, y, 1)  + T(z, xy, 1) - T(zx, y, 1) \Big) \\
    T(z, x, y)
    &= \frac{z}{1-z} + \frac{z}{1-y}\Big(yS(z, x, y) - S(z, xy, 1) + T(z, x, y) - T(z, xy, 1)\Big) \\
    F^\ddagger(z) &= S(z, 1, 1) + T(z, 1, 1).
\end{align}

Unfortunately, we have been unable to solve these equations (but do not expect $\cl{I}(\ne, >, \ge)$ to be algebraic, in any case).

\subsection{Class 2106: $(>, \le, \ge) \equiv (100, 101, 201)$}
\label{ssec:2106}

In this section we present a succession rule grown on the right for the class $\cl{I}(>, \le, \ge)$, but we do not believe this class to have an algebraic generating function. The construction for class 2106 differs from those presented in \cref{ssec:1176,ssec:1253,ssec:1016}: as the sequences grow, we can directly track the number of active sites (positions where a new last entry can be validly appended) as well as the maximum value. We note that our construction for $(100, 101, 201)$-avoiders is very similar to the construction for $(101, 201)$-avoiders presented by \citet{martinez_patterns_2018}, in which the statistics tracked are ``number of valid descents'' and ``number of valid ascents'', represented by $h$ and $k$ respectively.\footnote{We thank one of the reviewers for pointing this out.}

Suppose $a \in \cl{I}(100, 101, 201)$ is non-decreasing and ends on value $h > 0$, and we append the value $k < h$ to $a$. This inversion sequence may never return to any of the values in the interval $[k, h]$ to ensure avoidance of $(100, 101, 201)$. It may descend further to a value $j < k$ which augments the set of forbidden values to $[j, h]$, or ascend to $i > h$. Future descents require knowledge of which values are forbidden, or at least \emph{how many} values are forbidden. We therefore track the number of allowed values, along with the length and maximum value of each sequence.

Define the sets
\begin{itemize}
    \item $\cl{P}_{n, h, k}$ to be inversion sequences in $\cl{I}_n(>, \le, \ge)$ ending on their maximum $h$, which have $k$ valid choices for an immediate descent. An equivalent definition of $k$ is as the cardinality of the set \newline$\{i \in [0, h-1] \;|\; a \circ i \in \cl{I}(>, \le, \ge)\}$.
    \item $\cl{Q}_{n, h, k}$ to be inversion sequences in $\cl{I}_n(>, \le, \ge)$ ending below their maximum $h$, which have $k$ valid choices for an immediate descent.
\end{itemize}

Elements of $\cl{P}_{n, h, k}$ may ascend/repeat to any value $i \in [h, n]$, which adds $i-h$ valid future descents to the values in $[h, i-1]$. Elements of $\cl{Q}_{n, h, k}$ may ascend to any value $i \in [h+1, n]$ (since returning to $h$ produces 101), which adds $i-h-1$ valid future descents to the values in $[h+1, i-1]$. Any ascent must be to a new maximum value to avoid 201.

Now consider descents: suppose $a \in \cl{P}_{n, h, k} \cup \cl{Q}_{n, h, k}$ has valid descents to the values ${j_0 < \ldots < j_{k-1}}$. A descent to the value $j_i$ invalidates all future instances of $j_{i+1}, \ldots j_{k-1}$ to ensure avoidance of 201, leaving $i$ valid future descents to the values in $\{j_0,\dots, j_{i-1}\}$. Elements of $\cl{Q}$ may never repeat their final value, in order to avoid 100.

Assigning the labels $(n, h, k)_p$ to elements of $\cl{P}_{n, h, k}$ and $(n, h, k)_q$ to elements of $\cl{Q}_{n, h, k}$ allows us to write the following succession rule:
\begin{equation}\Omega_{(>, \le, \ge)}:\left\{\begin{aligned}
(0, 0, 0)_p \\
(n, h, k)_p &\leadsto (n+1, i, k+i-h)_p &&\text{for}\; i \in [h, n] \\
&\leadsto (n+1, h, i)_q &&\text{for}\; i \in [0, k-1] \\[1ex]
(n, h, k)_q &\leadsto (n+1, i, k+i-h-1)_p &&\text{for}\; i \in [h+1, n] \\
&\leadsto (n+1, h, i)_q &&\text{for}\; i \in [0, k-1].
\end{aligned}\right.\end{equation}

Using the generating functions $P(z,x,y)$ and $Q(z,x,y)$ with $z$, $x$, $y$ corresponding to $n$, $h$, and $k$ respectively, the rule $\Omega_{(>, \le, \ge)}$ gives equations
\begin{align}
    P(z, x, y) &= 1 + \frac{z}{1-xy}\Big[P(z, x, y) - xyP\left(zxy, \frac{1}{y}, y\right) + xQ(z, x, y) - xQ\left(zxy, \frac{1}{y}, y\right)\Big] \\
    Q(z, x, y) &= \frac{z}{1-y}[P(z, x, 1) - P(z, x, y) + Q(z, x, 1) - Q(z, x, y)].
\end{align}
The generating function for the class $\cl{I}(>, \le, \ge)$ is the sum of $P(z, x, y)$ and $Q(z, x, y)$. We have been unable to solve these equations.

\section{Inversion sequences grown on the left}\label{sec:left}

In this section we will discuss the classes for which growing on the left is simpler.

In every construction described here, it is always necessary to increment every non-zero term. It is also always necessary to prepend a 0 after incrementing terms. Moreover one can generally always just prepend a 0 without incrementing any other 0's (this is always accounted for in the succession rule). We will generally not mention these facts when describing a construction, and focus only on which 0's are being incremented. We also point out that in this section we will reuse symbols $\cl{A}$, $\cl{B}$, and so on, but the meanings never carry over from one subsection to the next unless explicitly stated.

\subsection{Class 663A: $(-,\neq,\geq) \equiv (010,101,110,120,201,210)$}
\label{ssec:663A}

\citet{martinez_patterns_2018} show that class 663A $(-,\neq,\geq)$ is Wilf-equivalent to class 663B $(\neq,-,\geq)$, but neither class is enumerated. Here we will give a succession rule and use it to solve the generating function.

Define $\cl{A} = \cl{I}(-,\neq,\geq)$, and define $\cl{C} \supset \cl{A}$ to be the set of those inversion sequences which avoid all the relevant patterns, except possibly occurrences of 010 where both 0's have value 0. Let $\cl{B} = \cl{C} \setminus \cl{A}$. Inversion sequences in $\cl{A}$ have an initial prefix of 0's and then all subsequent values are positive, while inversion sequences in $\cl{B}$ have exactly two consecutive runs of 0's (if there were more than two then a pattern 110, 120 or 210 would necessarily occur). For all objects in $\cl{C}$, let $P$ denote the initial run of 0's, let $S$ be the next run of 0's (if it exists), and let $p = |P|$.

For growing on the left, observe that
\begin{itemize}
    \item any 0's incremented must always be consecutive (or else 101 or 010 would occur);
    \item if any 0's in $S$ are incremented then they must all be incremented simultaneously (or else 201 or 210 would occur);
    \item if any 0's in $P$ are incremented then this must either be some number of 0's on the right, or else a single 0 not on the right;
    \item if $S$ exists then no 0's in $P$ can be incremented (or else a 120 would occur).
\end{itemize}
Using labels $(p)_a$ and $(p)_b$ for objects in $\cl{A}$ and $\cl{B}$ respectively, we have the rule
\begin{equation}
\Omega_{(-,\ne,\geq)}:\left\{\begin{aligned}
(0)_a \\
(p)_a &\leadsto (k)_a &&\text{for}\; k \in [1, p+1] \\
&\leadsto (\ell)_b &&\text{for}\; \ell \in [1, p-1] \\[1ex]
(p)_b &\leadsto (p+1)_a\\
&\leadsto (p+1)_b.
\end{aligned}\right.
\label{eqn:rule663A}
\end{equation}

\begin{figure}
\centering
\begin{subfigure}[t]{0.45\textwidth}
    \resizebox{\textwidth}{!}{\input{figures/5_A279553}}
    \caption{An inversion sequence in set $\cl{B}$ (\cref{ssec:663A}) with label $(4)_b.$}
\end{subfigure}
\hspace{4pt}
\begin{subfigure}[t]{0.45\textwidth}
    \resizebox{\textwidth}{!}{\input{figures/6_A279553}}
    \caption{The two possible progressions from the left sequence; the black sequence has label $(5)_b$, and the sequence with blue nodes instead has label $(5)_a$.}
\end{subfigure}
\caption{A progression of inversion sequences described in the rule $\Omega_{(-, \ne, \ge)}$. Each positive term is raised, some zeros may be raised, and a zero is prepended. In this case, only the sequence containing blue nodes avoids $(-, \ne, \ge)$.}
\label{fig:prog_663}
\end{figure}
See \cref{fig:prog_663} for a depiction of growth on the left according to this succession rule.
% , provided since this is the first such rule given in this work.

Using generating functions $A(z,x)$ and $B(z,x)$, where $z$ marks length and $x$ marks $p$, the above rule can be written as (omitting the $z$ arguments for brevity)
\begin{align}
    A(x) &= 1 + \frac{zx}{1-x}\left[A(1) - xA(x)\right] + zxB(x) \label{eqn:663A_A_eqn}\\
    B(x) &= \frac{z}{1-x}\left[xA(1)-A(x)\right] + z + zxB(x).\label{eqn:663A_B_eqn}
\end{align}
(The $+z$ in \eqref{eqn:663A_B_eqn} compensates for a $-z$ which comes from the $A$ terms.)
This system is straightforward to solve: one can use \eqref{eqn:663A_B_eqn} to eliminate $B$ from \eqref{eqn:663A_A_eqn}, and then a simple application of the kernel method yields the result. The solution is given in the following theorem.
\begin{thm}
The generating function for the class $\cl{I}(-,\neq,\geq)$ is given by
\begin{align}
    A(1) &= \frac{(X-1)(1-zX+z^2X)}{zX} \\
    &= 1 + z + 2 z^2 + 5 z^3 + 15 z^4 + 50 z^5 + 178 z^6 + 663 z^7 + 2552 z^8 + \dots
\end{align}
where $X \equiv X(z)$ is the only power series root of the kernel
\begin{equation}
    1 - x - zx + 2 zx^2 + z^2x  - z^2x^3.
\end{equation}
\end{thm}
$A(1)$ has minimal polynomial
\begin{equation}
    1 - (4-z)A +2(2-z)A^2 -A^3.
\end{equation}

\subsection{Class 1420: $(-,-,>) \equiv (100,110,120,201,210)$}
\label{ssec:1420}

At the start of this section we note that there is a solution to this sequence given on the OEIS, authored by Vladimir Kruchinin on March 25, 2019:
\begin{equation}
    F(x) = \frac{3}{4-4 \sin \left(\frac{1}{3} \arcsin\left(\frac{1}{16} (27 x+11)\right)\right)}.
\end{equation}

However no citation or other explanation is given. For this reason we will briefly describe a succession rule for this class.

Inversion sequences which avoid the triple $(-,-,>)$ have no non-consecutive occurrences of the pattern 10. They thus can only have two possible forms:
\begin{itemize}
    \item[(a)] $P \circ R$, where $P$ is a prefix of $p=|P|$ 0's and $R$ (possibly empty) contains no 0's; or
    \item[(b)] $P \circ (k) \circ (0) \circ R$, where $P$ is a prefix of $p=|P|$ 0's, $(k)$ is a single positive value, and $R$ (possibly empty) contains no values less than $k$.
\end{itemize}
Define sets $\cl{A}$ and $\cl{B}$ containing inversion sequences of these two types respectively. To grow these on the left, we have the following construction:
\begin{itemize}
    \item From $a$ in $\cl{A}$, we can increment any number of $0$'s on the right of $P$ (giving something in $\cl{A}$); we can then also increment the second-rightmost remaining 0, if there is one (giving something in $\cl{B}$). 
    % Finally we prepend a 0.
    \item From $b$ in $\cl{B}$, we can increment the rightmost 0 and any number of 0's on the right of $P$ (giving something in $\cl{A}$; we can then also increment the second-rightmost remaining 0, if there is one (giving something in $\cl{B}$).
    % Finally we prepend a 0. Alternatively we can just prepend a 0 without incrementing anything, giving something in $\cl{B}$.
\end{itemize}
This can be written as
\begin{equation}
\Omega_{(-,-,>)}:\left\{\begin{aligned}
(0)_a \\
(p)_a &\leadsto (p+1-i)_a &&\text{for}\; i \in [0, p] \\
&\leadsto (p+1-i)_b &&\text{for}\; i \in [2, p] \\[1ex]
(p)_b &\leadsto (p+1 - i)_a &&\text{for}\; i \in [0, p] \\
&\leadsto (p+1)_b \\
&\leadsto (p+1-i)_b &&\text{for}\; i \in [2, p].
\end{aligned}\right.
\label{eqn:rule1420}
\end{equation}
Using generating functions $A(z,x) \equiv A(x)$ and $B(z,x)\equiv B(x)$, where $z$ marks length and $x$ marks $p$, we get the pair of functional equations
\begin{align}
    A(x) &= 1+\frac{zx}{1-x}\left[A(1) - xA(x)\right] + \frac{zx}{1-x}\left[B(1)-xB(x)\right] \label{eqn:1420_Aeqn}\\
    B(x) &= \frac{z}{1-x}\left[xA(1)-A(x)\right] + z + \frac{z}{1-x}\left[xB(1)-B(x)\right]+zx B(x)\label{eqn:1420_Beqn}
\end{align}
Note that the `$+z$' in \eqref{eqn:1420_Beqn} compensates for the empty sequence in $\cl{A}$ which has $p=0$.

Eliminating $B(x)$ between these two equations gives something which can be readily solved using the kernel method. We find the following.
\begin{thm} The generating function for class $\cl{I}(-,-,>)$ is given by
\begin{align}
    F(z) &= A(1)+B(1) = \frac{(1+z)(X-1)-z(1-z)X^2}{z(1+z)X} \\
    &= 1 + z + 2 z^2 + 6 z^3 + 21 z^4 + 81 z^5 + 332 z^6 + 1420 z^7 + 
 6266 z^8 + 28318 z^9 + \dots \notag
\end{align}
where $X \equiv X(z) = 1+2z+5z^2+17z^3+64z^4+\dots$ is the only power series root of the kernel
\begin{equation}
    1-x+z-zx+2zx^2-z^2x^3.
\end{equation}
\end{thm}
$F(z)$ has minimal polynomial $-(z+1)F^3 + 4F^2 - 4F + 1$.

\subsection{Class 1833A: $(-,\neq,>) \equiv (110,120,201,210)$}
\label{ssec:1833A}

It was proved by \citet{martinez_patterns_2018} that classes 1833A $(-,\neq,>)$ and 1833B $(\neq,-,>)$ are Wilf-equivalent, but actual solutions to the counting sequences or generating functions were not found. Here we will present a construction for growing class 1833A on the left, and solve the resulting functional equations to obtain an algebraic generating function. We note as an aside that we have also found a construction for growing on the right, but it is somewhat more complicated.

First note that any inversion sequence in $\cl{I}(-,\neq,>)$ can have at most two consecutive runs of 0's, or else a pattern 110, 120 or 210 would be formed. We thus give each sequence the label $(p,s)$, where $p$ is the length of the first run of 0's and $s$ is the length of the second run of 0's (with $s=0$ if there is only one run of 0's). Call these runs of 0's the \emph{prefix} $P$ and \emph{suffix} $S$ (where we abuse the usual meaning of `suffix', since the sequence may actually end with a non-zero term).

Now observe that
\begin{itemize}
    \item[(i)] we cannot increment any 0's in $P$ before incrementing all 0's in $S$, or else the pattern 120 would be formed;
    \item[(ii)] we must increment all 0's in $S$ simultaneously, or else patterns 201 or 210 would be formed; and
    \item[(iii)] we can increment a consecutive run of 0's on the right of $P$, plus a single additional 0 not on the right (creating a non-empty suffix).
\end{itemize}
This gives the succession rule
\begin{equation}
\Omega_{(-,\neq,>)}:\left\{\begin{aligned}
(0,0) \\
(p,s) &\leadsto (p+1,s) \\
&\leadsto (p+1,0) &&\text{if } s>0\\
&\leadsto (p-\ell,k) &&\text{for}\; \ell \in [0,p-1], \; k \in[0,\ell]. \\
\end{aligned}\right.
\label{eqn:rule1833A}
\end{equation}
With $\alpha_{n,p,s}$ the number of inversion sequences of length $n$ with label $(p,s)$, and corresponding generating function $A(z,x,y) = \sum_{n,p,s}\alpha_{n,p,s}z^n x^p y^s$, the rule \eqref{eqn:rule1833A} can be encoded as the functional equation
\begin{multline}\label{eqn2:Axy}
    A(x,y) = 1 + zxA(x,y) + zx\left[A(x,1)-A(x,0)\right] + \frac{zx^2}{(y-x)(1-x)}A(x,1) \\ + \frac{zx}{(1-y)(1-x)}A(1,1) - \frac{zyx}{(y-x)(1-y)}A(y,1),
\end{multline}
where we have omitted the $z$ argument everywhere for brevity.

Set $y=0$ and use the resulting equation to eliminate $A(x,0)$, then go back and set $y=0$ and $x=y$ and use this to eliminate $A(y,1)$. We get
\begin{multline}\label{eqn:mainffe_1833A}
    0=-\frac{x z  \left(y x z-x^2z+1\right)}{(1-x) (y-x)}A(1,1)+\frac{x z  \left(y x^2 z-y x+y-x^3 z+x^2\right)}{(1-x) (y-x)}A(x,1)-(1-xz) A(x,y)\\ +\frac{x }{y (y-x)}A(y,0)-\frac{y^2 x z-y^2-y x^2 z+y x+x}{y (y-x)}
\end{multline}
Now we can set $y=1$, giving
\begin{multline}\label{eqn2:nox}
    0 = \frac{x z  \left(x^2 z-x z-1\right)}{(1-x)^2}A(1,1)-\frac{\left(x^4 z^2-x^3 z^2-2 x^3 z+3 x^2 z+x^2-2 x z-2 x+1\right) }{(1-x)^2}A(x,1)\\+\frac{x }{1-x}A(1,0)+\frac{x^2 z-x z-2 x+1}{1-x}.
\end{multline}
The coefficient of $A(x,1)$ has two roots which are Puiseux series in $z$. (They can be written with radicals but the expressions are very large.) Near $z=0$ they look like
\begin{align}
    X_1 &= 1 - z^{1/2} + z -2z^{3/2} + \frac72z^2-\frac{13}{2}z^{5/2} + 13z^3 + \dots \\
    X_2 &= 1 + z^{1/2} + z +2z^{3/2} + \frac72z^2+\frac{13}{2}z^{5/2} + 13z^3 + \dots
\end{align}
The other two roots diverge as $z\to0$.

We can substitute both of these into \eqref{eqn2:nox} to get two equations relating $A(1,1)$ and $A(1,0)$. These can then be solved. 
\begin{thm}
The generating function for class $\cl{I}(-,\neq,>)$ is given by
\begin{align}
    A(1,1) &= \frac{(X_1-1)(X_2-1) \left(zX_1X_2-1\right)}{X_1 X_2 z (X_1 X_2 z-X_1 z-X_2 z+z+1)} \\
    &= 1 + z + 2 z^2 + 6 z^3 + 22 z^4 + 90 z^5 + 396 z^6 + 1833 z^7 + 
 8801 z^8 + 43441 z^9 +\dots
\end{align}
\end{thm}

The minimal polynomial of $A(1,1)$ is
\begin{multline}
    (3z^3 + z)A^6 - (2z^3 + 6z^2 + 3z + 1)A^5 + (11z^2 + 5z + 5)A^4 \\ - (5z^2 + 7z + 10)A^3 + (6z + 10)A^2 - (2z + 5)A + 1.
\end{multline}

\subsection{Class 733: $(\neq,\neq,\geq) \equiv (010,101,120,201,210)$}
\label{ssec:733}

Class 733 is another class for which we have found constructions for growing on the left and on the right. Similar to class 1833A, we find growing on the left to be more straightforward and we will not explain the other construction here. The generating function for this class is also algebraic.

The construction for class 733 is similar to that of 1833A. Let $\cl{A} = \cl{I}(\neq,\neq,\geq)$, and let $\cl{B} \supset \cl{A}$ be those inversion sequences which avoid $(010,101,120,201,210)$, except possibly for the pattern 010 where the 0's have value 0. As with class 1833A, note that an inversion sequence in $\cl{B}$ can have at most two consecutive runs of 0's, or else a pattern 101, 120 or 210 would be formed. Thus define $P,S,p$ and $s$ as in \cref{ssec:1833A}.

We then make the following observations:
\begin{itemize}
    \item[(i)] we must increment all 0's in $S$ before incrementing anything in $P$, or else a pattern 120 or 010 would be formed;
    \item[(ii)] we must increment all 0's in $S$ simultaneously, or else patterns 201 or 210 would be formed;
    \item[(iii)] if $s=0$ then we can increment any consecutive run of 0's within $P$.
\end{itemize}
With label $(p,s)$, we thus have the succession rule
\begin{equation}
\Omega_{(\neq,\neq,\geq)}:\left\{\begin{aligned}
(0,0) \\
(p,s) &\leadsto (p+1,s) \\
&\leadsto (p+1,0) &&\text{if}\; s>0\\
(p,0) &\leadsto (p-\ell,k) &&\text{for}\; \ell \in [0,p-1], \; k \in[0,\ell]. \\
\end{aligned}\right.
\label{eqn:rule733}
\end{equation}

With $B(z,x,y)$ the generating function for $\cl{B}$, with $z$ marking length, $x$ marking $p$ and $y$ marking $s$, we can write the above succession rule as (omitting the $z$ argument for brevity)
\begin{multline}
    B(x,y) = 1 + zxB(x,y) + zx\left[B(x,1)-B(x,0)\right] + \frac{zx^2}{(1-x)(y-x)}B(x,0) \\ + \frac{zx}{(1-y)(1-x)}B(1,0)-\frac{zxy}{(1-y)(y-x)}B(y,0).
\end{multline}
The generating function for $\cl{I}(\neq,\neq,\geq)$ is then $B(1,0)$.

We solve this in a similar way to \eqref{eqn:mainffe_1833A}. First set $y=0$ and use the resulting equation to eliminate $B(x,0)$; then set $y=0$ and $x=y$ and use this to eliminate $B(y,0)$. Finally set $y=1$, giving
\begin{multline} \label{eqn:733_main_kernel_eqn}
0= -\frac{\left(x^4 z^2-x^3 z^2-2 x^3 z-x^2 z^2+3 x^2 z+x^2-x z-2 x+1\right) }{(1-x) (x z-x+1)}B(x,1) - \frac{x\left(x z^2-1\right)}{x z-x+1}B(1,0) \\ -\frac{x z }{1-x}B(1,1)+\frac{x^2 z-2 x+1}{x z-x+1}.
\end{multline}
The coefficient of $B(x,1)$ is quartic in $x$. Two of its roots are power series in $z$, namely
\begin{align}
    X_1 &= \frac{-\sqrt{-2 \left(\sqrt{5}-3\right) z^2-4 \left(\sqrt{5}+3\right) z+4}-\sqrt{5} z+z+2}{4 z} \\
    &= 1 +\frac{1}{2} \left(1+\sqrt{5}\right) z +\left(2+\sqrt{5}\right) z^2 +\left(7+3 \sqrt{5}\right) z^3 +\left(25+11 \sqrt{5}\right) z^4 + \dots\notag
\end{align}
and 
\begin{align}
    X_3 &= \frac{-\sqrt{2 \left(\sqrt{5}+3\right) z^2+4 \left(\sqrt{5}-3\right) z+4}+\sqrt{5} z+z+2}{4 z} \\
    &= 1 +\frac{1}{2} \left(1-\sqrt{5}\right) z +\left(2-\sqrt{5}\right) z^2 +\left(7-3 \sqrt{5}\right) z^3 +\left(25-11 \sqrt{5}\right) z^4 + \dots\notag
\end{align}
We can substitute both of these into \eqref{eqn:733_main_kernel_eqn} to give two equations relating $B(1,0)$ and $B(1,1)$. These can be solved to give the following.
\begin{thm}
The generating function for $\cl{I}(\neq,\neq,\geq)$ is given by
\begin{align}
    B(1,0) &= -\frac{X_1 z+X_3 z-1}{X_1 X_3 z^2-X_1 z-X_3 z-z+1} \\
    &= 1 + z + 2 z^2 + 5 z^3 + 15 z^4 + 51 z^5 + 188 z^6 + 733 z^7 + 
 2979 z^8 + 12495 z^9 + 53708 z^{10} + \dots
\end{align}
\end{thm}
The minimum polynomial of $B(1,0)$ is
\begin{equation}
    (4z^3 - 2z^2 + z)B^4 + (2z^3 - 7z^2 + z - 1)B^3 - (z^3 - 10z^2 + 3z - 3)B^2 - (z^2 + z + 3)B + 2z + 1.
\end{equation}

\subsection{Class 214: $(-,\geq,\geq) \equiv (000,010,100,110,120,210)$}
\label{ssec:214}

Given that inversion sequences in this class have to avoid a large number (6) of patterns, it is not surprising that the growth rule here is quite restrictive. Nevertheless we have not been able to solve the corresponding generating function and we do not believe it to be algebraic.

Let $\cl{A} = \cl{I}(-,\geq,\geq)$ and define $\cl{B} \supset \cl{A}$ to be the set of those inversion sequences which avoid all the necessary patterns, except possibly for 000, 010 and 100 when the 0's have value 0. Any inversion sequence in $\cl{B}$ can have at most two consecutive runs of 0's (call them $P$ and $S$), so we assign label $(p,s)$ where $p = |P|$ and $s = |S|$. We make the following observations:
\begin{itemize}
    \item if $s=0$ we can increment at most two 0's, and if we do increment two then one of these must be the rightmost 0;
    \item if $s>0$ then the only 0 we can increment is the rightmost one.
\end{itemize}
We get the following succession rule:
\begin{equation}
\Omega_{(-,\geq,\geq)}:\left\{\begin{aligned}
(0,0) \\
(p,s) &\leadsto (p+1,s) \\
&\leadsto (p+1,s-1) &&\text{if}\; s>0\\
(p,0) &\leadsto (p-\ell,\ell) &&\text{for}\; \ell \in [0,p-1] \\
&\leadsto(p-\ell-1,\ell) &&\text{for}\; \ell \in [0,p-2]
\end{aligned}\right.
\label{eqn:rule214}
\end{equation}
The inversion sequences in $\cl{A}$ are those with labels $(p,0)$ with $p \leq 2$.

We can write a fairly simple-looking functional equation satisfied by $B(z,x,y)$ (with $z$ marking length, $x$ marking $p$ and $y$ marking $s$):
\begin{multline}
    B(x,y) = 1 + zxB(x,y) +\frac{zx}{y}\left[B(x,y)- B(x,0)\right]  + \frac{zx}{y-x}\left[B(y,0)-B(x,0)\right] \\ + \frac{z}{y(y-x)}\left[xB(y,0)-yB(x,0)\right] + \frac{z}{y},
\end{multline}
where we have omitted the $z$ argument for brevity. Gathering like terms,
\begin{equation}
    0 = y+z - (y-zx-zxy)B(x,y) + \frac{xz(1+y)}{y-x}B(y,0) - \frac{z(y+2xy-x^2)}{y-x}B(x,0).
\end{equation}
However unfortunately we have been unable to solve this equation.

\subsection{Class 1509: $(-,\geq,>) \equiv (100,110,120,210)$}
\label{ssec:1509}

As with class 214, we can construct a fairly simple succession rule for $\cl{I}(-,\geq,>)$, but we have been unable to solve the corresponding generating function, and we do not believe it is algebraic. Let $\cl{A} = \cl{I}(-,\geq,>)$, and define $\cl{B} \supset \cl{A}$ to be the set of inversion sequences which avoid the patterns $(100,110,120,210)$, except possibly for the pattern 100 where the 0's take value 0.

Once again inversion sequences in $\cl{B}$ can have at most two consecutive runs of 0's. Define $P,S,p$ and $s$ as per the previous sections. We have the following rules about incrementing 0's:
\begin{itemize}
    \item if $s=0$ we can increment any number (including zero) of 0's on the right of $P$, and/or one other 0 in $P$;
    \item if $s=1$ the 0 in $S$ must be incremented, and then we can increment any number (including zero) of 0's on the right of $P$, and/or one other 0 in $P$;
    \item if $s>1$ then the only 0 we can increment is the rightmost one.
\end{itemize}
Assigning label $(p,s)$ as with per the previous sections, we have the rule
\begin{equation}
\Omega_{(-,\geq,>)}:\left\{\begin{aligned}
(0,0) \\
(p,s) &\leadsto (p+1,s) \\
&\leadsto (p+1,s-1) &&\text{if}\; s>0\\
&\leadsto (p-i,0) &&\text{if}\; s=0 \text{ or }1, \text{ for } i \in [0,p-1] \\
&\leadsto(p-\ell,\ell-k) &&\text{if}\; s=0 \text{ or }1, \text{ for } \ell \in [1,p-1], \; k \in [0,\ell-1].
\end{aligned}\right.
\label{eqn:rule1509}
\end{equation}
The elements of $\cl{A}$ are then those inversion sequences with label $(p,0)$ or $(p,1)$.

We can write a functional equation satisfied by the generating function for $\cl{B}$, but we are unable to solve it. Define $B_{[1]}(x) = [w^1]B(x,w)$ and similarly $B_{[\leq1]}(x) = B(x,0) + B_{[1]}(x)$. Then we have
\begin{multline}\label{eqn:1509_big_eqn}
    B(x,y) = 1 + zx B(x,y) + \frac{zx}{y}\left[B(x,y) - B(x,0)\right] + \frac{zx}{1-x}\left[B_{[\leq1]}(1)-B_{[\leq1]}(x)\right] \\ - \frac{zxy}{(1-x)(x-y)}B_{[\leq1]}(x) + \frac{zxy}{(1-x)(1-y)}B_{[\leq 1]}(1) + \frac{zxy}{(1-y)(x-y)}B_{[\leq1]}(y).
\end{multline}
Using the fact that
\begin{equation}
    B(x,0) = 1 + zxB(x,0) + zxB_{[1]}(x) + \frac{zx}{1-x}\left[B_{[\leq1]}(1)-B_{[\leq1]}(x)\right]
\end{equation}
we can rewrite \eqref{eqn:1509_big_eqn} as
\begin{multline}
    B(x,y) = B(x,0) - zxB_{[1]}(x) + \frac{zx(y+1)}{y}\left[B(x,y) - B(x,0)\right] \\ - \frac{zxy}{(1-x)(x-y)}B_{[\leq1]}(x) + \frac{zxy}{(1-x)(1-y)}B_{[\leq 1]}(1) + \frac{zxy}{(1-y)(x-y)}B_{[\leq1]}(y).
\end{multline}

\subsection{Class 1953A: $(-,>,>) \equiv (110,120,210)$}
\label{ssec:1953A}

This class was shown by \citet{martinez_patterns_2018} to be Wilf-equivalent to class 1953B $(100,120,210)$, but no enumeration was given. Here we give a succession rule for growing this class on the left. A fairly compact functional equation for the generating function can be given, but we are unable to solve it and we do not believe it to be algebraic.

As with the classes in the previous sections, inversion sequences in this class can have at most two consecutive runs of 0's. We again define $P,S,p$ and $s$ as before. This time:
\begin{itemize}
    \item we cannot increment any 0's in $P$ before incrementing all 0's in $S$;
    \item the 0's in $S$ must be incremented from right to left (one or more at a time);
    \item if we increment all 0's in $S$ (or if $s=0$) then we can increment any number of 0's on the right of $P$, plus one additional 0 somewhere in $P$.
\end{itemize}
With label $(p,s)$, we thus get the succession rule
\begin{equation}
\Omega_{(-,>,>)}:\left\{\begin{aligned}
(0,0) \\
(p,s) &\leadsto (p+1,s-i) &&\text{for }\; i \in [0,s] \\
&\leadsto (p+1-\ell,k) &&\text{for}\; \ell\in[1,p], \; k \in[0,\ell-1].
\end{aligned}\right.
\label{eqn:rule1953A}
\end{equation}
Here every label $(p,s)$ is associated with a valid member of $\cl{I}(-,>,>)$. This succession rule converts to the following functional equation satisfied by $B(z, x, y)$ (with $z$ marking length, $x$ marking $p$ and $y$ marking $s$):
% \begin{multline}
%     B(x, y) = 1 + \frac{zx}{1-y}B(1, 0) - \frac{zxy}{1-y}B(x, y) + \frac{zxy}{(1-y)(y-x)}B(x, 0) - \frac{zxy}{(1-y)(y-x)}B(y, 0),
% \end{multline}
\begin{multline}
    B(x,y) = 1 - \frac{zxy}{1-y}B(x,y) - \frac{xz(x^2+y-2xy)}{(1-x)(x-y)(1-y)}B(x,1) \\ + \frac{zxy}{(1-y)(x-y)}B(y,1) + \frac{zx}{(1-x)(1-y)}B(1,1)
\end{multline}
where we omit the argument $z$ for brevity.

\subsection{Class 759: $(\leq, \neq, \geq) \equiv (010,110,120)$}
\label{ssec:759}

Let $\cl{A}$ be the class of inversion sequences avoiding the triple $(\leq, \neq, \geq)$. It will be convenient to also define $\cl{B} \supset \cl{A}$ to be the set of inversion sequences which avoid $(010,110,120)$, except possibly for occurrences of 010 where the 0's have value 0. For $a \in \cl{A}$ (or $\cl{B}$), let $P$ be the initial run of 0's and let $p = |P|$. We make the following observations:
\begin{itemize}
    \item If $a$ avoids 010, all the 0's of $a$ must be in $P$. However, during the process of growing $a$ on the left, there may have been intermediate stages from $\cl{B}\setminus \cl{A}$.
    \item If $a \in \cl{A}$, then we can grow on the left by:
    \begin{itemize}
        \item[(i)] incrementing some number of 0's on the right of $P$ (giving something in $\cl{A}$),
        \item[(ii)] possibly incrementing one additional 0 not on the right of the remaining 0's (giving something in $\cl{B}$, which will later need to be converted back to $\cl{A}$),
        \item[(iii)] finally prepending a 0.
    \end{itemize}
    \item If $b \in \cl{B}$ then we must increment all the 0's not in $P$ (in some order) before we can increment anything in $P$ (otherwise we would create a 110 or 120 pattern). The order can be anything so long as we do not create a 010, 110 or 120 pattern.
\end{itemize}
At this point we make use of the idea of Pantone's \emph{commitments} \cite{Pantone}. When performing action (ii) above, suppose there are $k\geq1$ 0's which are not in the new $P$. At that moment \emph{we can determine ahead of time} in which order those $k$ 0's will be incremented. We do this by assigning each of them a positive integer, according to the following conditions:
\begin{itemize}
    \item[R1.] Each integer $1,2,\dots,b$ is used at least once, for some $b\in[1,k]$.
    \item[R2.] The resulting word of length $k$ on the alphabet $\{1,\dots,b\}$ avoids the patterns 212, 112 and 213.
\end{itemize}
Here the assignment of integers determines the order of incrementing: those 0's assigned 1 will be incremented first, followed by those assigned 2, and so on.
See \cref{fig:759_commitments} for an example.

\begin{figure}
\centering
    \begin{subfigure}{0.30\textwidth}
        \centering
        \input{figures/7_759_part1}
        \caption{}
    \end{subfigure}
    \begin{subfigure}{0.30\textwidth}
        \centering
        \input{figures/8_759_part2}
        \caption{}
    \end{subfigure}
    \begin{subfigure}{0.30\textwidth}
        \centering
        \input{figures/9_759_part3}
        \caption{}
    \end{subfigure}
    \caption{An illustration of the growth of Class 759 using commitments. The sequence in (a) is in $\cl{A}$, with $p=7$. In (b) we have incremented a 0 not on the left of $P$, to get something in $\cl{B}$ with $k=6$. The number of commitments is $c=4$, with the word on the alphabet $\{1,2,3,4\}$ avoiding 212, 112 and 213. In (c) we have fulfilled the first commitment.}
    \label{fig:759_commitments}
\end{figure}

We get the following lemma, whose proof can be found in \cref{sec:comb_proofs}.
\begin{lem}\label{lem:words_R1R2}
    The number of words of length $k$ which satisfy conditions R1 and R2 above is
    \begin{equation}
        a_{k,b} = \binom{k-1}{k-b}C_b,
    \end{equation}
    where $C_b$ is the $b^\text{th}$ Catalan number.
\end{lem}

We can now construct a succession rule for inversions sequences $S \in \cl{B}$ which have been assigned a commitment, i.e.\ a pattern-avoiding word $w$ as described above. To each such pair $(S,w)$  we will assign a label $(p,c)$. Here $c$ is the number of commitments remaining, i.e.\ the number of remaining steps in which we will increment the 0's not in $P$. If an inversion sequence is in $\cl{A} \subset \cl{B}$, it gets label $(p,0)$ for some $p\geq0$. If $c>0$ then our only choices are to fulfill a commitment (decreasing $c$ by 1) or not (leaving $c$ as is), as well as prepend a 0. 

If we write $\ell$ to be the number of 0's incremented in step (i) above plus $k$, then we have the succession rule
\begin{equation}
\Omega_{(\leq,\neq,\geq)}: \left\{\begin{aligned} (0,0) \\ (p,c) &\leadsto (p+1,c)\\ &\leadsto (p+1,c-1) && \text{if}\; c>0 \\
(p,0) &\leadsto (p-\ell,b)^{m_{\ell,b}} && \text{for}\; 0 \leq \ell \leq p-1,\, 0 \leq b \leq \ell, \, p \geq 1 \end{aligned} \right.
\end{equation}
where
\begin{equation} 
m_{\ell,b} = \sum_{k=b}^\ell a_{k,b} = \binom{\ell}{b}C_b.
\end{equation}
% Here the members of $\cl{A} = \cl{I}(\leq,\neq,\geq)$ are those objects with label $(p,0)$ for some $p$.

This succession rule closely resembles the rule $\Omega_{120}'$ from \citet{testart_generating_2025}, which served as inspiration. We have not attempted to rewrite $\Omega_{(\leq,\neq,\geq)}$ in terms of generating functions. As we will see in \cref{ssec:asymps_nonalg}, the asymptotics are quite complex, and we have no reason to think that the corresponding generating function is algebraic, or even D-finite.

\subsection{Class 247: $(\leq, -, \geq) \equiv (000,010,110,120)$}
\label{ssec:247}

This class is quite similar to class 759. Again we let $\cl{A}$ be the class of inversion sequences avoiding the triple of interest $(\leq, -, \geq)$, and we let $\cl{B}$ be the set of inversion sequences avoiding $(000,010,110,120)$, except possibly for occurrences of 000 or 010 where the 0's have value 0. $P$ and $p$ are as defined previously. The growth for this class is now
\begin{itemize}
    \item If $a \in \cl{A}$, we grow on the left by incrementing at most two 0's (if we increment two then one of these must be the rightmost 0), and then prepending a 0.
    \item If $b\in \cl{B}$ we must increment all the 0's not in $P$, in some order, which can be anything so long as we do not create a 000, 010, 110 or 120 pattern.
\end{itemize}
We can now implement commitments in the same way that we did for class 759. We still follow condition R1 from \cref{ssec:759}, but now R2 is replaced by
\begin{itemize}
    \item[R3.] The resulting word of length $k$ on the alphabet $\{1,\dots,b\}$ avoids the patterns 111, 212, 112 and 213.
\end{itemize}
We have the following lemma, which is again proved in \cref{sec:comb_proofs}.
\begin{lem}\label{lem:words_R1R3}
    The number of words of length $k$ which satisfy conditions R1 and R3 is
    \begin{equation}
        d_{k,b} = \binom{b}{k-b}C_b,
    \end{equation}
    where $C_b$ is the $b^\text{th}$ Catalan number.
\end{lem}
We can again construct a succession rule for pairs $(S,w)$, where $S\in\cl{B}$ and $w$ is a pattern-avoiding word. We then assign a label $(p,c)$ to each such pair, where $p$ and $c$ mean the same thing as in \cref{ssec:759}. However now if an inversion sequence is in $\cl{A}$, it will have label $(p,0)$ with $p \in \{0,1,2\}$. We have the rule
\begin{equation}
\Omega_{(\leq, -, \geq)}: \left\{\begin{aligned} (0,0) \\ (p,c) &\leadsto (p+1,c)\\ &\leadsto (p+1,c-1) && \text{if}\; c>0 \\
(p,0) &\leadsto (p-\ell,b)^{w_{\ell,b}} && \text{for}\; 0 \leq \ell \leq p-1,\, \left\lceil{\textstyle \frac{\ell-1}{2}}\right\rceil \leq b \leq \ell, \, p \geq 1 \end{aligned} \right.
\end{equation}
where
\begin{equation} 
w_{\ell,b} = d_{\ell-1,b} + d_{\ell,b} = \binom{b+1}{\ell-b}C_b.
\end{equation}

We do not expect this to have a simple formulation in terms of generating functions, and our analysis of the sequence suggests that it does not have an algebraic or D-finite generating function.

\section{Asymptotics}\label{sec:asymptotics}

In this section we will briefly discuss the asymptotic behaviour of the counting sequences for each class covered in \cref{sec:right,sec:left}. For some of these classes, precise asymptotics have already been given in the corresponding OEIS pages, and we simply quote those results -- when this is the case we indicate it below.

\subsection{Algebraic classes} \label{ssec:asymps_alg}

We first summarise the asymptotics for all the classes with algebraic generating functions. These are easily derived from the dominant singularities (see e.g.\ \citet{flajolet_analytic_2009}) and we will not go into details here.

\begin{itemize}
    \item Class 663A $(-,\neq,\geq)$ (\href{https://oeis.org/A279553}{A279553}, \cref{ssec:663A}): (as per OEIS page)
    \begin{equation} I_n \sim \frac{C}{\sqrt{\pi}} \cdot n^{-3/2} \cdot \mu^n, \end{equation}
    where
    \begin{align*}
        \mu &\approx 4.73508 \text{ is a root of } 4-12x+4x^2-24x^3+5x^4 \\
        C &\approx 0.391676 \text{ is a root of } 5 + 64y^2 + 33728y^4 - 209664y^6 - 93184y^8.
    \end{align*}
    
    \item Class 733 $(\neq,\neq,\geq)$ (\href{https://oeis.org/A279554}{A279554}, \cref{ssec:733}):
    \begin{equation} I_n \sim \frac{C}{\sqrt{\pi}} \cdot n^{-3/2} \cdot \mu^n, \end{equation}
    where
    \begin{align*}
        \mu &\approx 5.16207 \text{ is a root of } 1 - 14 x + 7 x^2 - 6 x^3 + x^4 \\
        C &\approx 0.178902 \text { is a root of a large degree 16 polynomial}
    \end{align*}

    \item Class 1016 $(>,-,\neq)$ (\href{https://oeis.org/A279560}{A279560}, \cref{ssec:1016}): (as per OEIS page)
    \begin{equation} I_n \sim \frac{1}{4\sqrt{\pi}} \cdot n^{-1/2} \cdot 4^n \end{equation}

    \item Class 1176 $(>,\leq,\neq)$ (\href{https://oeis.org/A279562}{A279562}, \cref{ssec:1176}):
    \begin{equation} I_n \sim \frac{C}{\sqrt{\pi}} \cdot n^{-3/2} \cdot \left(2+2\sqrt{2}\right)^n \end{equation}
    where
    \[ C \approx 0.716315 \text{ is a root of } 961 - 41296 x^2 + 76832 x^4.\]

    \item Class 1253 $(>,\neq,\neq)$ (\href{https://oeis.org/A279563}{A279563}, \cref{ssec:1253}): (as per OEIS page)
    \begin{equation} I_n \sim \frac{1}{3\sqrt{\pi}} \cdot n^{-1/2} \cdot 4^n \end{equation}

    \item Class 1420 $(-,-,>)$ (\href{https://oeis.org/A279565}{A279565}, \cref{ssec:1420}): (as per OEIS page)
    \begin{equation} I_n \sim \frac{\sqrt{15}}{8\sqrt{2\pi}} \cdot n^{-3/2} \cdot \left(27/5\right)^n \end{equation}

    \item Class 1833A $(-,\neq,>)$ (\href{https://oeis.org/A279568}{A279568}, \cref{ssec:1833A}): (as per OEIS page)
    \begin{equation} I_n \sim \frac{C}{\sqrt{\pi}} \cdot n^{-3/2} \cdot \mu^n \end{equation}
    where
    \begin{align*}
        \mu &\approx 5.98042 \text{ is a root of } 32 + 195 x + 12 x^2 + 112 x^3 - 20 x^4 \\
        C &\approx 0.187339 \text{ is an algebraic number (which we have been unable to compute exactly)}
    \end{align*}
\end{itemize}

\subsection{Non-algebraic classes} \label{ssec:asymps_nonalg}

In this section we discuss the classes which we believe do not have algebraic generating functions. The asymptotic expressions in this section are non-rigorous and are based only on numerical analysis. There are several classes which appear to have algebraic-type dominant singularities, and we summarise those first.

\begin{itemize}
    \item Class 214 $(-,\geq,\geq)$ (\href{https://oeis.org/A279544}{A279544}, \cref{ssec:214}): (as per OEIS page)
    \begin{equation} I_n \sim C \cdot n^{-3/2} \cdot 4^n \end{equation}
    where $C \approx 0.0549097$.

    \item Class 830 $(\neq,>,\geq)$ (\href{https://oeis.org/A279558}{A279558}, \cref{ssec:830}): (thanks to V\'aclav Kot\v{e}\v{s}ovec for pointing out how to use Richardson extrapolation to get a precise estimate for $C$)
    \begin{equation}
        I_n \sim C \cdot n^{-3/2} \cdot (27/4)^n
    \end{equation}
    where $C \approx 0.00018$. 

    \item Class 1509 $(-,\geq,>)$ (\href{https://oeis.org/A279567}{A279567}, \cref{ssec:1509}): (as per OEIS page)
    \begin{equation} I_n \sim C \cdot n^{-3/2} \cdot (3+2\sqrt{2})^n \end{equation}
    where $C \approx 0.0660857.$

    \item Class 1953A $(-,>,>)$ (\href{https://oeis.org/A279569}{A279569}, \cref{ssec:1953A}): (as per OEIS page)
    \begin{equation} I_n \sim C \cdot n^{-3/2} \cdot (27/4)^n \end{equation}
    where $C \approx 0.0111684$.

\end{itemize}

For the remaining classes the asymptotic behaviour of $I_n$ is more complicated. Our best estimates are:
\begin{itemize}
    \item Class 247 $(\leq,-,\geq)$ (\href{https://oeis.org/A279551}{A279551}, \cref{ssec:247}): (thanks to Tony Guttmann for assistance with this estimate, as well for the next two classes -- see \citet{guttmann_analysis_2015} for some details on the methodology used)
    \begin{equation} I_n \sim C \cdot n^{g} \cdot 8^n \cdot \mu_1^{n^{3/8}} \end{equation}
    where $g \approx 4.25$, $\log \mu_1 \approx -13$ and $\log C \approx 8$.

    \item Class 759 $(\leq,\neq,\geq)$ (\href{https://oeis.org/A279556}{A279556}, \cref{ssec:759}):
    \begin{equation} I_n \sim C \cdot n^{g} \cdot 9^n \cdot \mu_1^{n^{3/8}} \end{equation}
    where $g \approx 3.25$, $\log \mu_1 \approx -10.4$ and $C$ is a constant.

    \item Class 2106 $(>,\leq,\geq)$ (\href{https://oeis.org/A279571}{A279571}, \cref{ssec:2106}):
    \begin{equation}
        I_n \sim C \cdot n^{g} \cdot 9^n
    \end{equation}
    where $g \approx -5.40169$ and $C \approx 88.7$.
\end{itemize}

See \cref{sec:numerical_details} for a brief discussion of the numerical work for Classes 830, 247 and 2106. The methodology for Class 759 was very similar to that of Class 247.

\section{Open questions and future work}\label{sec:conclusion}

This work enumerates all uncounted classes of inversion sequences avoiding a triple of relations, providing generating functions where (perceived to be) possible. Avoidance of a triple of relations always corresponds to avoiding a set of patterns of length 3. A natural extension of this work is to examine all uncounted classes of inversion sequences avoiding a set of three patterns of length 3. There are between 137 and 139 Wilf classes  of these (probably 137), described by \citet{callan_inversion_2023}. We believe that 76 of these Wilf classes have counting sequences which do not yet appear in the OEIS; moreover some of those which are in the OEIS do not appear to have any relevant citations listed. Completing this enumeration would thus be quite a monumental task, although constructions and succession rules for many of these classes will probably be quite similar to those presented in this work and others. Inversion sequences which avoid four (or more) patterns of length 3 can also be studied using the methods presented here.

% \nathan{, of which there appear to be THIS NUMBER}. We expect the constructions for many classes to be similar to those presented in this work and others.

The classes of inversion sequences avoiding a single pattern of length of 4 are also of interest, and appear to be (in most cases) significantly more complicated than their length-3 counterparts. To use binary relations, we note that in general $\binom{4}{2}=6$ relations are required to compare 4 elements of an inversion sequence.

We note that Classes 247 and 759 have asymptotic behaviour that is quite different from the other 12 classes (see \cref{ssec:asymps_nonalg} above). In particular the (conjectured) $\mu_1^{n^{3/8}}$ term (a so-called ``stretched exponential'') is reminiscent of the conjectured behaviour of 1324-avoiding permutations, see for example \citet{conway_1324-avoiding_2018,bevan_structural_2020}. The counting sequence for 1324-avoiding permutations is believed to have a term like $\mu_1^{\sqrt{n}}$ in its dominant asymptotics. It is very much an open question as to what determines whether a particular class of pattern-avoiding inversion sequences or permutations have stretched exponential behaviour.

The estimated exponent $g=-5.401$ for Class 2106 also stands out. If it is in fact a rational number then it may not have a small denominator. If it is instead irrational then that would be quite unusual.

\section*{Acknowledgements}
We thank Tony Guttmann and V\'aclav Kot\v{e}\v{s}ovec for assistance with estimates for some of the asymptotics. We also thank the anonymous reviewers for their comments on an earlier version of the paper. N.\ Britt was supported by a Small Grant from the School of Mathematics and Statistics at the University of Melbourne.

\sloppy
%\printbibliography
\bibliography{inversion_sequences}
\fussy

\appendix

\crefalias{section}{appendix}

\section{Combinatorial proofs for Classes 759 and 247}\label{sec:comb_proofs}

\cref{lem:words_R1R2,lem:words_R1R3} are fairly elementary, but we have not been able to find results like these in the literature, so we present brief proofs here.

\begin{proof}[of \cref{lem:words_R1R2}]
Let $w$ be a word of length $k$ on the alphabet $\{1,\dots,b\}$ which satisfies conditions R1 and R2. Each letter of the alphabet occurs somewhere in $w$. If we delete all repetitions of letters we are left with a $b$-permutation which avoids 213, of which there are $C_b$. Equivalently, if we look at the first occurrences of each letter in $w$, we get a 213-avoiding permutation.

Suppose we have fixed the ordering of these first occurrences of each letter. We then have $k-b$ remaining letters to choose and position. Note that in order to avoid the pattern 212, all the occurrences of $b$ must be consecutive. Then, to avoid 112 and 212, all occurrences of $b-1$ must be consecutive and come after the occurrences of $b$, except possibly for the first $b-1$. Similarly, all occurrences of $b-2$ must be consecutive and come after $b$ and $b-1$, except possibly for the first occurrence, and so on.

It follows that once we have chosen how many times each letter will be repeated, there is only one way to position them. The number of choices for the repetitions is $\binom{b+k-b-1}{k-b} = \binom{k-1}{k-b}$, and the result follows.
\end{proof}

\begin{proof}[of \cref{lem:words_R1R3}]
This follows in the same way as the proof of \cref{lem:words_R1R2}. Because we now also have to avoid 111, we can choose at most one repetition of each letter $1,\dots,b$. The rules for how to place them remain the same, i.e.\ once the ordering of the first occurrence of each letter has been selected (there are $C_b$ ways to do this), there is only one way to place the repetitions. 

The number of ways to choose the repetitions is $\binom{b}{k-b}$, and the result follows.
\end{proof}

\section{Asymptotic analyses}\label{sec:numerical_details}

In this Appendix we briefly discuss the numerical methods used to estimate the asymptotics for some of the non-algebraic classes studied in this paper.

\subsection{Class 830}

For Class 830 (\href{https://oeis.org/A279558}{A279558}, \cref{ssec:830}), a plot of successive ratios $r_n=I_{n+1}/I_n$ against $\frac1n$ indicates an exponential growth rate of $\mu = \frac{27}{4}$, which is quite believable since there are other classes with this same growth rate (e.g. 1953A and B). See \cref{fig:830_numerics} (a).

\begin{figure}
    \centering
    \begin{subfigure}{0.49\textwidth}
        \includegraphics[width=\textwidth]{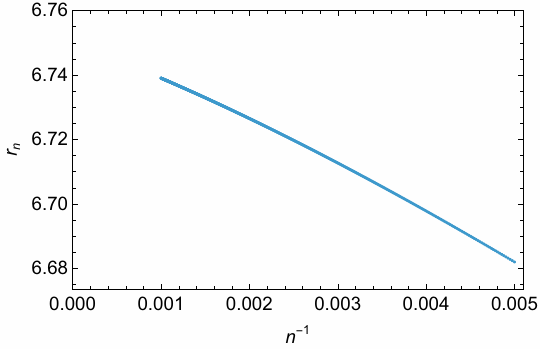}
        \caption{}
    \end{subfigure}
    \hfill
    \begin{subfigure}{0.49\textwidth}
        \includegraphics[width=\textwidth]{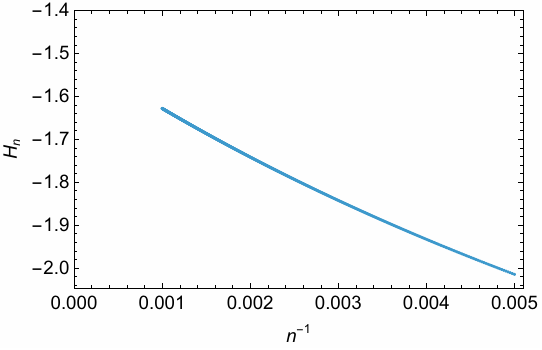}
        \caption{}
    \end{subfigure}
    \caption{(a) A plot of the ratios $r_n$ against $n^{-1}$ for Class 830. (b) A plot of the estimators $H_n$ against $n^{-1}$ for Class 830.}
    \label{fig:830_numerics}
\end{figure}

Then if we assume $I_n \sim C \cdot n^g \cdot \mu^n$ for constants $C,g$, we have
\begin{equation}
    \frac{r_n}{\mu} \sim \left(\frac{n+1}{n}\right)^g \sim 1+\frac{g}{n}. 
\end{equation}
So we expect $H_n = n\left(\frac{r_n}{\mu}-1\right) \to g$ as $n\to\infty$. Plotting this quantity against $\frac1n$ indicates that $g$ is very close to $-1.5$, and we thus work with the assumption $g=-\frac32$. See \cref{fig:830_numerics} (b).

Finally to accurately estimate $C$, we make the assumption that $I_n$ has the regular expansion
\begin{equation}
    I_n = C\cdot n^{-3/2} \cdot (27/4)^n \left(1 + \frac{C_1}{n} + \frac{C_2}{n^2} + \dots \right)
\end{equation}
Then letting $F(n) = \frac{n^{3/2}I_n}{(27/4)^n}$, we can apply Richardson extrapolation (see for example Section 8.1 in \cite{bender_advanced_1999}). This involves taking expansions of $m$ different terms $F(n_1), \dots, F({n_m})$ each up to order $1/n^{m-1}$, and eliminating the unknowns $C_1,\dots,C_{m-1}$ between them, leaving an estimate for $C$. A particularly fruitful implementation (we thank V\'aclav Kot\v{e}\v{s}ovec for pointing this out) uses
\begin{equation}
    n_j = j\cdot \left\lfloor \frac{n_\text{max}}{m}\right\rfloor,
\end{equation}
where $n_\text{max}$ is length of the known series for $I_n$. That is, we take $m$ evenly spaced values over the range $1,\dots,n_\text{max}$. Here we use $n_\text{max} = 1000$. The resulting estimate for $C$ ends up as
\begin{equation}
    \tilde C = \sum_{j=1}^m (-1)^{m+j} F(n_j) \frac{j^m}{j!(m-j)!}
\end{equation}
Using $m=10,20,\dots,100$ we have very good convergence, and estimate
\begin{equation}
    C = 0.0001809637754651179\dots
\end{equation}

\subsection{Class 247}

For Class 247 (\href{https://oeis.org/A279551}{A279551}, \cref{ssec:247}), plots of ratios $r_n = I_{n+1}/I_n$ against $1/n^p$ for various values of $p$ suggests a growth rate of $\mu=8$. (See \cref{fig:247_numerics} (a).) However, unlike Class 830, $H_n = n\left(\frac{r_n}{\mu}-1\right)$ does not appear to be approaching a constant as $n\to\infty$. This indicates that we do not have simple asymptotics of the form $C\cdot n^g \cdot \mu^n$, so instead we guess that there may also be a ``stretched exponential'' factor of the form $\mu_1^{n^\sigma}$ with $0 < \sigma < 1$. These sorts of expressions have been studied by Guttmann \cite{guttmann_analysis_2015}, and we follow some of the exploratory methods discussed there. Note that we had $n_\text{max} = 1625$ terms to work with.

\begin{figure}
    \centering
    \begin{subfigure}{0.49\textwidth}
        \includegraphics[width=0.97\textwidth]{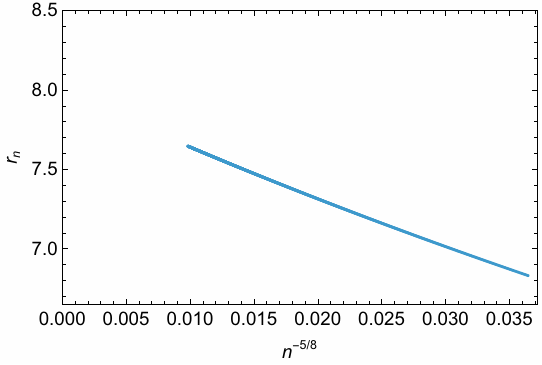}
        \caption{}
    \end{subfigure}
    \hfill
    \begin{subfigure}{0.49\textwidth}
        \includegraphics[width=\textwidth]{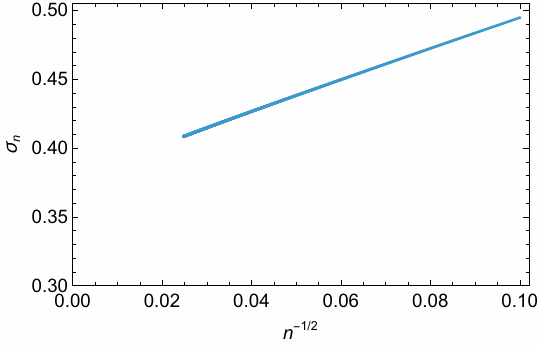}
        \caption{}
    \end{subfigure}
    \caption{(a) A plot of the ratios $r_n$ against $n^{-5/8}$ for Class 247. (b) A plot of the estimators $\sigma_n$ against $n^{-1/2}$ for Class 247.}
    \label{fig:247_numerics}
\end{figure}

First, we can expand $r_n$ to get
\begin{equation}\label{eqn:stretched_ratio_expand}
r_n = \mu\left(1+\frac{\sigma \log\mu_1}{n^{1-\sigma}} + \frac{g}{n} + \frac{\sigma^2\log^2\mu_1}{2n^{2-2\sigma}} + O(n^{\sigma-2})\right).
\end{equation}
(Note that this ordering of terms assumes that $\sigma<\frac12$.)
There are several ways to estimate $\sigma$ (see \cite{guttmann_analysis_2015}). One is to see from \eqref{eqn:stretched_ratio_expand} that
\begin{equation}
    s_n = n^{1-\sigma}\left(\frac{r_n}{\mu}-1\right) = \sigma \log\mu_1 + \frac{g}{n^\sigma} + \frac{\sigma^2\log^2\mu_1}{2n^{1-\sigma}} + O(n^{-1}).
\end{equation}
By testing different values of $\sigma$, we should thus be able to find either a $\sigma \in (0,\frac12)$ such that a plot of $s_n$ against $1/n^\sigma$ appears linear, or a value of $\sigma \in (\frac12,1)$ such that a plot of $s_n$ against $1/n^{1-\sigma}$ is linear. In this case it is the former that works, and we estimate that $\sigma \approx \frac38 = 0.375$. An alternative method is to let
\begin{equation}
    \ell_n = \log\left|\frac{r_n}{\mu}-1\right|
\end{equation}
which, following from \eqref{eqn:stretched_ratio_expand}, should be asymptotically linear in $\log n$ with gradient $\sigma-1$. Accordingly we can compute local gradients from consecutive terms to get an estimate for $\sigma$:
\begin{equation}
    \sigma_n = 1 + \frac{\ell_n - \ell_{n-1}}{\log n-\log(n-1)}
\end{equation}
Doing so again gives $\sigma \approx \frac38$. See \cref{fig:247_numerics} (b).

Estimating $C, g$ and $\mu_1$ is somewhat more crude. We simply take the logarithm of the leading asymptotics,
\begin{equation}
    \log I_n \sim \log C + n\log \mu + g \log n + n^\sigma \log \mu_1,
\end{equation}
and substitute in our conjectured values $\mu=8$ and $\sigma=\frac38$. Then we take three consecutive terms $n-1,n,n+1$ and fit to the assumed form, solving for $\log C, g$ and $\log \mu_1$. Taking these estimates and plotting against $\frac{1}{n^\sigma}$, we estimate $g\approx 4.25$,  $\log\mu_1 \approx -13$, and $\log C \approx 8$.

\subsection{Class 2106}

The approach for Class 2106 (\href{https://oeis.org/A279571}{A279571}, \cref{ssec:2106}) resembles that of Class 830, although we do not achieve nearly the same precision. Again starting with a plot of ratios $r_n = I_{n+1}/I_n$, we find $\mu = 9$. Assuming that $I_n \sim C\cdot n^g \cdot \mu^n$, we then consider $H_n = n\left(\frac{r_n}{\mu}-1\right)$, but find that it is not obviously approaching any rational number with a small denominator. 

Applying Richardson extrapolation to the sequence $H_n$ using $n_\text{max}=1000$ terms gives $g \approx -5.40169$. Then taking the sequence
\begin{equation}
    F(n) = \frac{I_n}{9^n n^g}
\end{equation}
and either plotting against $\frac1n$ or using Richardson extrapolation up to around $m=25$ (larger values of $m$ lead to numerical instability), we estimate $C \approx 88.7$.

\end{document}

%% file: figures/1_A279562.tex
%A (102, 201)-avoiding sequence in set \cl{C}

\begin{tikzpicture}[mnode/.style={circle,draw=black,fill=black,inner sep=0pt,minimum size=4pt}]
            \draw[step=0.5cm, gray,thin] (-0.5,0) grid (8.4,2.9);
            \draw (0,0) node[mnode]{}
                \foreach \x/\y in {0/0, 0.5/0.5, 1/0.5, 1.5/1, 2/1, 2.5/2.5, 3/2.5, 3.5/2.5}{
            -- (\x,\y) node[mnode]{}
            };
            \draw[->] (-0.5, 0) -- (8.5, 0);
            \draw[->] (-0.5, 0) -- (-0.5, 3);
            
\end{tikzpicture}

%% file: figures/2_A279562.tex
%A (102, 201)-avoiding sequence in set \cl{C}

\begin{tikzpicture}[mnode/.style={circle,draw=black,fill=black,inner sep=0pt,minimum size=4pt}]
            \draw[step=0.5cm, gray,thin] (-0.5,0) grid (8.4,2.9);
            \draw (0,0) node[mnode]{}
                \foreach \x/\y in {0/0, 0.5/0.5, 1/0.5, 1.5/1, 2/1, 2.5/2.5, 3/2.5, 3.5/2.5, 4/2}{
            -- (\x,\y) node[mnode]{}
            };
            \draw[->] (-0.5, 0) -- (8.5, 0);
            \draw[->] (-0.5, 0) -- (-0.5, 3);
            
\end{tikzpicture}

%% file: figures/3_A279562.tex
%A (102, 201)-avoiding sequence in set \cl{D}

\begin{tikzpicture}[mnode/.style={circle,draw=black,fill=black,inner sep=0pt,minimum size=4pt}]
            \draw[step=0.5cm, gray,thin] (-0.5,0) grid (8.4,2.9);
            \draw (0,0) node[mnode]{}
                \foreach \x/\y in {0/0, 0.5/0.5, 1/0.5, 1.5/1, 2/1, 2.5/2.5, 3/2.5, 3.5/2.5, 4/2, 4.5/2.5, 5/2.5, 5.5/2.5, 6/2.5}{
            -- (\x,\y) node[mnode]{}
            };
            \draw[->] (-0.5, 0) -- (8.5, 0);
            \draw[->] (-0.5, 0) -- (-0.5, 3);
            
\end{tikzpicture}

%% file: figures/4_A279562.tex
%General (102, 201)-avoiding sequence in set \cl{E}

\begin{tikzpicture}[mnode/.style={circle,draw=black,fill=black,inner sep=0pt,minimum size=4pt}]
            \draw[step=0.5cm, gray,thin] (-0.5,0) grid (8.4,2.9);
            \draw (0,0) node[mnode]{}
                \foreach \x/\y in {0/0, 0.5/0.5, 1/0.5, 1.5/1, 2/1, 2.5/2.5, 3/2.5, 3.5/2.5, 4/2, 4.5/2.5, 5/2.5, 5.5/2.5, 6/2.5, 6.5/1.5, 7/1, 7.5/0.5, 8/0}{
            -- (\x,\y) node[mnode]{}
            };
            \draw[->] (-0.5, 0) -- (8.5, 0);
            \draw[->] (-0.5, 0) -- (-0.5, 3);
            
\end{tikzpicture}

%% file: figures/5_A279553.tex
%A (010, 100, 101, 120, 201, 210)-avoiding sequence

\begin{tikzpicture}[mnode/.style={circle,draw=black,fill=black,inner sep=0pt,minimum size=4pt}]
            \draw[step=0.5cm, gray,thin] (-0.5,0) grid (5.9,2.9);
            \draw (0,0) node[mnode]{}
                \foreach \x/\y in {0/0, 0.5/0, 1/0, 1.5/0, 2/1, 2.5/0, 3/0, 3.5/0, 4/1.5, 4.5/1.5, 5/2}{
            -- (\x,\y) node[mnode]{}
            };
            \draw[->] (-0.5, 0) -- (6, 0);
            \draw[->] (-0.5, 0) -- (-0.5, 3);
            
\end{tikzpicture}

%% file: figures/6_A279553.tex
%A (010, 100, 101, 120, 201, 210)-avoiding sequence with progressions depicted

\begin{tikzpicture}[mnode/.style={circle,draw=black,fill=black,inner sep=0pt,minimum size=4pt}, pnode/.style={circle,draw=blue,fill=blue,inner sep=0pt,minimum size=4pt}]
            \draw[step=0.5cm, gray,thin] (-0.5,0) grid (5.9,2.9);
                        \draw (2.5,1.5) node[pnode]{}
                \foreach \x/\y in {2.5/1.5, 3/0.5, 3.5/0.5, 4/0.5, 4.5/2}{
            -- (\x,\y) node[pnode]{}
            };
            \draw (0,0) node[mnode]{}
                \foreach \x/\y in {0/0, 0.5/0, 1/0, 1.5/0, 2/0, 2.5/1.5, 3/0, 3.5/0, 4/0, 4.5/2, 5/2, 5.5/2.5}{
            -- (\x,\y) node[mnode]{}
            };
            
            \draw[->] (-0.5, 0) -- (6, 0);
            \draw[->] (-0.5, 0) -- (-0.5, 3);
            
\end{tikzpicture}

%% file: figures/7_759_part1.tex
\begin{tikzpicture}[mnode/.style={circle,draw=black,fill=black,inner sep=0pt,minimum size=4pt}, scale=0.7]
            \draw[step=0.5cm, gray,thin] (-0.5,0) grid (5.4,2.9);
            \draw (0,0) node[mnode]{}
                \foreach \x/\y in {0/0, 0.5/0, 1/0, 1.5/0, 2/0, 2.5/0, 3/0, 3.5/0.5, 4/1.5}{
            -- (\x,\y) node[mnode]{}
            };
            \draw[->] (-0.5, 0) -- (5.5, 0);
            \draw[->] (-0.5, 0) -- (-0.5, 3);

            \node at (1,0) [below] {\phantom{\footnotesize \color{blue}{1}}};
            
\end{tikzpicture}

%% file: figures/8_759_part2.tex
\begin{tikzpicture}[mnode/.style={circle,draw=black,fill=black,inner sep=0pt,minimum size=4pt}, scale=0.7]
            \draw[step=0.5cm, gray,thin] (-0.5,0) grid (5.4,2.9);
            \draw (0,0) node[mnode]{}
                \foreach \x/\y in {0/0, 0.5/0.5, 1/0, 1.5/0, 2/0, 2.5/0, 3/0, 3.5/0, 4/1, 4.5/2}{
            -- (\x,\y) node[mnode]{}
            };
            \draw[->] (-0.5, 0) -- (5.5, 0);
            \draw[->] (-0.5, 0) -- (-0.5, 3);

            \node at (1,0) [below] {\footnotesize \color{blue}{4}};
            \node at (1.5,0) [below] {\footnotesize \color{blue}{1}};
            \node at (2,0) [below] {\footnotesize \color{blue}{2}};
            \node at (2.5,0) [below] {\footnotesize \color{blue}{3}};
            \node at (3,0) [below] {\footnotesize \color{blue}{3}};
            \node at (3.5,0) [below] {\footnotesize \color{blue}{1}};
            
\end{tikzpicture}

%% file: figures/9_759_part3.tex
\begin{tikzpicture}[mnode/.style={circle,draw=black,fill=black,inner sep=0pt,minimum size=4pt}, scale=0.7]
            \draw[step=0.5cm, gray,thin] (-0.5,0) grid (5.4,2.9);
            \draw (0,0) node[mnode]{}
                \foreach \x/\y in {0/0, 0.5/0, 1/1, 1.5/0, 2/0.5, 2.5/0, 3/0, 3.5/0, 4/0.5, 4.5/1.5, 5/2.5}{
            -- (\x,\y) node[mnode]{}
            };
            \draw[->] (-0.5, 0) -- (5.5, 0);
            \draw[->] (-0.5, 0) -- (-0.5, 3);

            \node at (1.5,0) [below] {\footnotesize \color{blue}{4}};
            % \node at (2,0) [below] {\footnotesize \color{blue}{1}};
            \node at (2.5,0) [below] {\footnotesize \color{blue}{2}};
            \node at (3,0) [below] {\footnotesize \color{blue}{3}};
            \node at (3.5,0) [below] {\footnotesize \color{blue}{3}};
            % \node at (4,0) [below] {\footnotesize \color{blue}{1}};
            
\end{tikzpicture}